\documentclass[a4paper,reqno]{amsart} 

\usepackage{amssymb}
\usepackage[mathcal]{euscript} 

\usepackage{tikz}
\usetikzlibrary{shapes.geometric}

\usepackage{tikz-cd} 
\usetikzlibrary{babel}

\usepackage{hyperref}

\usepackage{bbm}

\newcommand{\bydef}{:=}

\newcommand{\id}{\mathrm{id}}%identity map
\newcommand{\cC}{\mathcal{C}}

\newcommand{\cK}{\mathcal{K}}

\newcommand{\cQ}{\mathcal{Q}}

\newcommand{\ZZ}{\mathbb{Z}}

\newcommand{\FF}{\mathbb{F}} 

\newcommand{\Buno}{\mathbbm{1}} 
\DeclareMathOperator{\Hom}{\mathrm{Hom}}
\DeclareMathOperator{\End}{\mathrm{End}}

\DeclareMathOperator{\im}{\mathrm{im}\,}
\DeclareMathOperator{\coker}{coker}

\newcommand{\subo}{_{\bar 0}} 
\newcommand{\subuno}{_{\bar 1}}
\newcommand{\Bup}{\textup{B}} 
\newcommand{\Qup}{\textup{Q}}

\newcommand{\Repe}{\mathsf{Rep\,C}_3}

\newcommand{\vect}{\mathsf{Vec}}
\newcommand{\sVec}{\mathsf{sVec}}
\newcommand{\balpha}{\boldsymbol{\alpha}}

\newcommand{\Repap}{\mathsf{Rep}\,\balpha_p}
\newcommand{\Ver}{\mathsf{Ver}}

\newtheorem{theorem}{Theorem}[section]
\newtheorem{proposition}[theorem]{Proposition}
\newtheorem{lemma}[theorem]{Lemma}
\newtheorem{corollary}[theorem]{Corollary}

\theoremstyle{definition} 
\newtheorem{definition}[theorem]{Definition}

\theoremstyle{remark} \newtheorem{remark}[theorem]{Remark}
\numberwithin{equation}{section}

\def\hreglon{\hrule height 1pt}
\def\vreglon{\vrule height 12pt width1pt depth 4pt}

\begin{document}

\title%[]%
{Composition algebras in symmetric tensor categories}

\author[A.~Elduque]{Alberto Elduque} 
\address{Departamento de
Matem\'{a}ticas e Instituto Universitario de Matem\'aticas y
Aplicaciones, Universidad de Zaragoza, 50009 Zaragoza, Spain}
\email{elduque@unizar.es} 
\thanks{Both authors have been supported by grant
PID2021-123461NB-C21, funded by 
MCIN/AEI/ 10.13039/501100011033 and by
 ``ERDF A way of making Europe''. A.E. also acknowledges support by grant 
E22\_20R (Gobierno de Arag\'on), while J.R-I. acknowledges support by
grant Fortalece 2023/03 funded by ``Comunidad Aut\'onoma 
de La Rioja''; and by a predoctoral research grant FPI-2023 funded by ``Universidad de La Rioja''.}

\author[J.~R\'andez-Ib\'a\~nez]{Javier R\'andez-Ib\'a\~nez}
\address{Departamento de Matem\'aticas y Computaci\'on,
Universidad de La Rioja,\phantom{26006 26006} 26006 Logro{\~n}o, Spain}
\email{javier.randez@unirioja.es}



%\date{}

\begin{abstract}
Composition algebras in symmetric tensor categories are defined. The classical results on standard unital 
composition algebras, as well as the classification of the unital composition superalgebras, are reviewed in 
light of this definition, and the unital composition algebras in the category
$\Ver_4^+$ (the counterpart to the category of vector superspaces in characteristic $2$) are classified.
\end{abstract}

\maketitle

%-----------------------

\bigskip

This paper is devoted to the definition of composition algebras in symmetric tensor categories.
This is relatively straightforward if the characteristic is not $2$, and was considered in \cite{DES}, but requires
extra care if characteristic $2$ is included, as quadratic forms, considered in \cite{Kannanetal}, are more
restrictive than symmetric bilinear forms. As a consequence, the definition of composition algebras in
Definition \ref{df:compoSTC} is subtler than the usual definition in the category of vector spaces.

The known classifications of compostion algebras in the categories of vector spaces or superspaces will be reviewed
under our more general definition, and the unital composition algebras will be classified in the category
$\Ver_4^+$ which is the counterpart to the category of vector superspaces in characteristic $2$.

%-----------------------
\bigskip

\section{Introduction}

Symmetric tensor categories constitute a natural setting to study algebras, extending the setting of 
the usual algebras, which are algebras in the category of vector spaces, or the superalgebras, which are algebras
in the category of vector superspaces. 

The celebrated Deligne's Theorem \cite{Deligne1,Deligne2} asserts
that any symmetric tensor category of moderate growth over an algebraically closed field of characteristic $0$ fibers
over the category of finite-dimensional vector superspaces. This is no longer true in characteristic $p>0$. 
A counterexample is provided (for $p>3$) by the Verlinde category $\Ver_p$, which arises as the semisimplification
of either the category of representations of the cyclic group of order $p$ or, alternatively, as the 
semisimplification of the category of representations of the Frobenius kernel $\balpha_p$ of the additive group
scheme $\mathbb{G}_a$ (see \cite{EtingofOstrik}). 

The process of semisimplification allowed Kannan \cite{Kannan} to obtain  the exceptional contragredient Lie
superalgebras specific of characteristics $3$ and $5$ by choosing suitable
degree $p$ nilpotent derivations on exceptional simple Lie algebras. This allowed him to consider these exceptional
simple Lie algebras as Lie algebras in the category $\Repap$, for $p=3$ or $5$, and then, by semisimplification,
he obtained Lie algebras in $\Ver_3\simeq \sVec$ or in a suitable subcategory of $\Ver_5$ equivalent to
$\sVec$. Many of these exceptional Lie superalgebras appear in a \emph{super magic square} in
\cite{CunhaElduque}, constructed in terms of composition superalgebras. In \cite{DES} it is shown how the
exceptional composition superalgebras can be constructed, in turn, through semisimplification of the 
split Cayley algebra,
providing a new point of view to some of the cases in \cite{Kannan}. Also, another case in \cite{Kannan} in
characteristic $5$ is explained in \cite{EEK} by obtaining the $10$-dimensional simple Jordan superalgebra of
Kac (see \cite{Kac,HK,BE}) through semisimplification of the split Albert algebra (exceptional Jordan algebra).
The process of semisimplification of contragredient Lie algebras has been treated with great generality
in \cite{APS}.

In characteristic $3$, the Verlinde category $\Ver_3$ is equivalent to the category $\sVec$ of finite-dimensional
vector superspaces, while in characteristic $2$, $\Ver_2$ is equivalent to the category $\vect$ of finite-dimensional
vector spaces. On the other hand, the category of vector superspaces in characteristic $2$ is just the category
of $\ZZ_2$-graded vector spaces, as the parity signs make no sense. In \cite{Venkatesh} it is  suggested that 
that the right analog of $\sVec$ in characteristic $2$ is the category $\Ver_4^+$ reviewed below, which does
not fiber over $\Ver_2$.

 \medskip
 
Let us review some facts that will be used throughout.

Given an arbitrary field $\FF$ of characteristic $2$, $\Ver_4^+$ is, as a tensor category, the category of
representations of the Hopf algebra of dual numbers $\FF[t]=\FF 1\oplus\FF t$, where $t^2=0$, and $t$ is primitive,
while the braiding arises from equipping $\FF[t]$ with the triangular structure given by the $R$-matrix 
$R=1\otimes 1+t\otimes t$ (see, e.g., \cite[\S 8.3]{EGNO}). 
In other words, given two $\FF[t]$-modules $U$ and $V$, the braiding $c_{U,V}$ is given by
\[
\begin{split}
c_{U,V}\colon U\otimes V&\longrightarrow V\otimes U\\
 u\otimes v&\mapsto v\otimes u + t(v)\otimes t(u).
\end{split}
\]
Lie algebras in the category $\Ver_4^+$ have been considered in \cite{BP,Hu}.

\medskip

Let $\FF$ be an arbitrary ground field. A \emph{composition algebra} over $\FF$ is a nonassociative (i.e., not
necessarily associative) algebra $C$ over $\FF$ (in other words, a vector space with a bilinear map
$C\times C\rightarrow C$), endowed with a regular quadratic form $\Qup\colon C\rightarrow \FF$,
called the \emph{norm}, which is multiplicative: $\Qup(xy)=\Qup(x)\Qup(y)$ for all $x,y\in C$.
In case $C$ is unital, $C$ is said to be a \emph{unital composition algebra} or a \emph{Hurwitz algebra}.
In this case $\Qup(1)=1$.
The celebrated Hurwitz Theorem  asserts that the dimension of a Hurwitz algebra is restricted to 
$1$, $2$, $4$, or $8$. These algebras are the analogues, over arbitrary fields, to the real division algebras of
the real numbers, complex numbers, quaternions and octonions (see, e.g., \cite{Eld00}).

Any unital composition algebra satisfies the Cayley-Hamilton equation of degree $2$:
\begin{equation}\label{eq:compo}
x^2-\Bup(x,1)x+\Qup(x)1=0
\end{equation}
for all $x\in C$, where $\Bup(x,y)\bydef \Qup(x+y)-\Qup(x)-\Qup(y)$ denotes the polar form of the quadratic
form $\Qup$.

Unital composition algebras are behind many exceptional situations in both Algebra and Geometry (see, e.g.
\cite{Bae02}). In particular, they are behind the exceptional simple Lie and Jordan algebras. Hence, 
in order to study classes of interesting algebras in symmetric tensor categories, 
the study of the unital composition algebras
is a natural place to start.

Actually, we will need later on an extension of the unital composition algebras. 
A unital nonassociative algebra $C$, endowed
with a quadratic form $\Qup\colon C\rightarrow \FF$ (no regularity condition is imposed), is said to be 
\emph{conic} if $\Qup(1)=1$ and \eqref{eq:compo} is satisfied (see \cite[\S 16.1]{GPR}).

\medskip

The goal of this work is the definition and study of the unital composition algebras in symmetric tensor categories,
over an arbitrary field $\FF$.

\medskip

The paper is organized as follows. In Section \ref{se:quadratic}, the definition and some properties of the
quadratic forms over symmetric tensor categories will be reviewed. These have been considered in \cite{Kannanetal},
and special attention must be paid to the case of characteristic $2$.

The definition of the (unital) composition algebra in a symmetric tensor category is given in 
Definition \ref{df:compoSTC} in Section \ref{se:composition}. In this section, the linearizations of the 
multiplicative property of the norm $\Qup$ of such an algebra are derived. While this is quite
straightforward in the category $\vect$, using elements, it is much more subtle in the generality considered here.
If the characteristic is not $2$, or if the characteristic is $2$ but the \emph{Frobenius twist} is trivial, then 
the associated symmetric bilinear form $\Bup$ determines $\Qup$ and the multiplicative property is implied
by the full linearization of it (Proposition \ref{pr:full_multiplicative}).

Section \ref{se:examples} reviews the known classification results on standard unital composition algebras 
(in $\vect$) and unital composition superalgebras, and how they fit in our more general theory. The definition
of unital composition superalgebra in \cite{EOsuper} does not coincide exactly, 
in characteristic $2$, with our definition 
here of a unital composition algebra in the symmetric tensor category $\sVec$, and these subtleties require
some care.

Finally, Section \ref{se:ver4} deals with the unital composition algebras in the category $\Ver_4^+$.
In case the Frobenius twist is trivial, there appears a family of two-dimensional unital composition algebras, and
another family of four-dimensional unital composition algebras (Theorem \ref{th:kertimtclass}). 
On the other hand, if the Frobenius twist is not trivial, the classification of the unital composition algebras
reduces to the well-known classification of the unital composition algebras in $\vect$.

%-------------------------

\bigskip

\section{Quadratic forms in symmetric tensor categories}\label{se:quadratic}

Quadratic forms in symmetric tensor categories have been dealt with in \cite{Kannanetal}. Here we will review
the results there in a way suitable for our purposes.

Given an object $U$ in a symmetric tensor category $\cC$ over a field $\FF$, consider the kernel 
$(\Gamma^2(U),\iota_U)$
of the morphism $\id-c_{U,U}\colon U\otimes U\rightarrow U\otimes U$. By its own definition, the composition
\[
\Gamma^2(U)\xrightarrow{\iota_U} U\otimes U\xrightarrow{\id -c_{U,U}}U\otimes U
\]
is trivial, so $c_{U,U}\circ\iota_U=\iota_U$.

Since $c_{U,U}^2=\id$, the composition $(\id -c_{U,U})\circ(\id+c_{U,U})$ is trivial, so there is a unique
morphism $\rho_U\colon U\otimes U\rightarrow \Gamma^2(U)$ such that
\begin{equation}\label{eq:rhoU}
\iota_U\circ\rho_U=\id +c_{U,U}.
\end{equation}
That is, we have the commutative diagram:
\[
\begin{tikzcd}
U\otimes U \arrow[rr, "\id+c_{U,U}"] \arrow[dr, "\rho_U"', dashrightarrow] 
& & U\otimes U \arrow[rr, "\id-c_{U,U}"] & & U\otimes U\\
& \Gamma^2(U)  \arrow[ur, "\iota_U", rightarrowtail]
\end{tikzcd}
\]
As $\iota_U\circ\rho_U\circ c_{U,U}=(\id+c_{U,U})\circ c_{U,U}=(\id+c_{U,U})=\iota_U\circ\rho_U$ and 
$\iota_U$ is a monomorphism, we get
\begin{equation}\label{eq:rhoUc}
\rho_U=\rho_U\circ c_{U,U}.
\end{equation}

\begin{definition}\label{df:QB}
Let $U$ be an object in a symmetric tensor category $\cC$.
\begin{itemize}
\item A \emph{quadratic form} on $U$ is a morphism $\Qup\colon \Gamma^2(U)\rightarrow \Buno$.
\item Given a quadratic form $\Qup$ on $U$, its \emph{associated bilinear form}  is the 
morphism $\Bup=\Qup\circ\rho_U\colon U\otimes U\rightarrow \Buno$. (This is symmetric, i.e. 
$\Bup=\Bup\circ c_{U,U}$,
because of \eqref{eq:rhoUc}.)
\item A quadratic form $\Qup$ on $U$ is said to be \emph{nondegenerate} if so is its associated bilinear form
$\Bup$, that is (\cite[Definition 2.3]{Kannanetal}), if the image of $\Bup$ under the isomorphism
\[
\Hom_{\cC}(U\otimes U,\Buno)\simeq\Hom_{\cC}(U,U^*)
\]
afforded by tensor-hom adjunction, is an isomorphism.
\end{itemize}
\end{definition}

The definition of the associated bilinear form is given in a different way in \cite[Definition 2.5]{Kannanetal},
although we will see soon that both definitions are equivalent (Remark \ref{re:B_Kannan}).

The case of characteristic $2$ is always more subtle. To begin with, the morphism $\rho_U$ in \eqref{eq:rhoU}
is always an epimorphism if the characteristic is not $2$, due to the next proposition.

\begin{proposition}\label{pr:rhoUnot2}
Let $U$ be an object in a symmetric tensor category $\cC$ over a field $\FF$. 
With the notations above, the following equation hold:
\[
\rho_U\circ\iota_U=2\id_{\Gamma^2(U)}.
\]
In particular, if the characteristic of $\FF$ is not $2$, then $\rho_U$ is an epimorphism. 
\end{proposition}
\begin{proof}
Note that $\iota_U\circ\rho_U\circ\iota_U=(\id_{U\otimes U}+c_{U,U})\circ\iota_U
=2\iota_U=\iota_U\circ(2\id_{\Gamma^2(U)})$, so the result follows because $\iota_U$ is a monomorphism.

Alternatively, if the characteristic is not $2$, 
the result follows from the fact that $\frac{\id-c_{U,U}}{2}$ and $\frac{\id+c_{U,U}}{2}$ 
are orthogonal
idempotents in $\End_\cC(U\otimes U)$ whose sum is the identity. Then $\Gamma^2(U)$, which is the
kernel of $\id - c_{U,U}$,  and hence also of the idempotent $\frac{\id-c_{U,U}}{2}$, is a direct summand
of $U\otimes U$, and the uniqueness of $\rho_U$ in \eqref{eq:rhoU} shows that $\rho_U$ is twice the projection
on this direct summand. The result follows.
\end{proof}

\begin{corollary}\label{co:Balternating}
Let $\Qup$ be a quadratic form on an object $U$ of a symmetric tensor category $\cC$ over a field
of characteristic $2$. Then the associated bilinear form $\Bup$ is alternating (i.e. $\Bup\circ\iota_U=0$).
\end{corollary}
\begin{proof}
The result follows from $\Bup\circ\iota_U=\Qup\circ\rho_U\circ\iota_U$ and this is $0$ if the characteristic is $2$
by Proposition \ref{pr:rhoUnot2}.
\end{proof}

On the other hand, if the characteristic of the ground field $\FF$ of the symmetric tensor category $\cC$ is $2$,
then $\delta\bydef \id-c_{U,U}=\id+c_{U,U}$ satisfies $\delta^2=0$. We have that
$\bigl(\Gamma^2(U),\iota_U\bigr)$ is the kernel of $\delta$. Let $(S^2(U),\pi_U)$ be the cokernel of $\delta$, and
let $\varphi\colon \Gamma^2(U)\rightarrow S^2(U)$ be the composition of $\iota_U$ and $\pi_U$.

The \emph{Frobenius twist} $U^{(1)}$ of our object $U$ is defined to be the image of $\varphi$ (see
\cite[\S 2.1.2]{Kannanetal}).

\begin{proposition}\label{pr:rhoU2}
Let $U$ be an object in a symmetric tensor category $\cC$ over a field $\FF$ of characteristic $2$. Then
$\rho_U$ is an epimorphism if and only if the Frobenius twist $U^{(1)}$ is trivial.
\end{proposition}
\begin{proof}
Let us consider, for a while, a more general situation. Let $\cC$ be any abelian category, and 
$\delta\colon W\rightarrow W$ an endomorphism in $\cC$ with $\delta^2=0$. Let $\ker \delta=(K,\iota)$,
$\coker\delta=(C,\pi)$, and consider the canonical decomposition
\[
\begin{tikzcd}
K \arrow[r, "\iota", rightarrowtail] & W \arrow[r, "p", twoheadrightarrow] & I \arrow[r, "i", rightarrowtail]
& W \arrow[r, "\pi", twoheadrightarrow] & C
\end{tikzcd}
\]
where $\delta=i\circ p$, $(K,\iota)=\ker\delta=\ker p$, $(C,\pi)=\coker\delta=\coker i$, $I$ is the image
of $\delta$, and also $(I,i)=\ker \pi$ and $(I,p)=\coker\iota$.

As $\delta^2=i\circ p\circ i\circ p=0$, while $i$ is a monomorphism and $p$ is an epimorphism, we get $p\circ i=0$
and hence
\begin{itemize}
\item using that $(K,\iota)$ is the kernel of $p$, it follows that there is a unique morphism 
$j\colon I\rightarrow K$ such that
$i=\iota\circ j$. Besides, $j$ is a monomorphism because so is $i$;

\item using that $(C,\pi)$ is the cokernel of $i$, it follows that there is a unique morphism
$q\colon C\rightarrow I$ such that $p=q\circ \pi$. Besides, $q$ is an epimorphism because so is $p$.
\end{itemize}

Therefore, we get the commutative diagram
%\[
%\begin{tikzcd}
%W \arrow[r, "\pi", twoheadrightarrow] \arrow[rrrr, bend right, looseness=0.6, "\delta", twoheadrightarrow] 
%&C \arrow[r, "q", twoheadrightarrow] & I \arrow[r, "j", rightarrowtail]
%& K \arrow[r, "\iota", rightarrowtail] & W
%\end{tikzcd}
%\]
\[
\begin{tikzcd}
W \arrow[r, "\pi", twoheadrightarrow] 
  \arrow[rrrr, 
         to path={
           [rounded corners=4pt] -- ([yshift=-.5cm]\tikztostart.center) 
           -- node[below] {$\scriptstyle\delta$} ([yshift=-.5cm]\tikztotarget.center) 
           -- (\tikztotarget)
         }
        ] 
& C \arrow[r, "q", twoheadrightarrow] 
& I \arrow[r, "j", rightarrowtail]
& K \arrow[r, "\iota", rightarrowtail] 
& W
\end{tikzcd}
\]
Consider now the morphim $\varphi=\pi\circ\iota\colon K\rightarrow C$. Then we get $(I,j)=\ker\varphi$.
Indeed, $\varphi\circ j=\pi\circ\iota \circ j=\pi\circ i=0$, and for any morphism $h\colon X\rightarrow K$ such
that $\varphi\circ h=0$, we get $\pi\circ\iota\circ h=0$, so that there is a unique morphism $k\colon X\rightarrow I$
such that $\iota\circ h=i\circ k$ (because $\ker\pi=(I,i)$). Then $\iota\circ h=i\circ k=\iota\circ j\circ k$, and
we conclude, using that $\iota$ is a monomorphism, that $h=j\circ k$.

In the same vein, we get $(I,q)=\coker\varphi$.

In particular, $\varphi$ is trivial if and only if $j$ is an isomorphism, if and only if $q$ is an isomorphism.

\smallskip

Let us go back to our original situation, where $\cC$ is a symmetric tensor category. Take $W=U\otimes U$ 
and $\delta=\id+c_{U,U}$,
so that $(K,\iota)=(\Gamma^2(U),\iota_U)$ and $(C,\pi)=(S^2(U),\pi_U)$. Thus we have the commutative
diagram
\[
\begin{tikzcd}
U\otimes U \arrow[r, "\pi_U", twoheadrightarrow] 
 \arrow[rrrr, 
         to path={
           [rounded corners=4pt] -- ([yshift=-.5cm]\tikztostart.center) 
           -- node[below] {$\scriptstyle{\id + c_{U,U}}$} ([yshift=-.5cm]\tikztotarget.center) 
           -- (\tikztotarget)
         }
        ] 
&S^2(U) \arrow[r, "q", twoheadrightarrow] & I \arrow[r, "j", rightarrowtail]
& \Gamma^2(U) \arrow[r, "\iota_U", rightarrowtail] & U\otimes U
\end{tikzcd}
\]
The uniqueness of $\rho_U$ in \eqref{eq:rhoU} forces $\rho_U=j\circ q\circ \pi_U$.
But $q$ and $\pi_U$ are epimorphisms and $j$ is a monomorphism. Hence $\rho_U$ is an epimorphism if and only
if $j$ is an isomorphism, and this happens if and only if the morphism $\varphi=\pi_U\circ\iota_U$ is trivial,
this being equivalent to the Frobenius twist $U^{(1)}$ being trivial.
\end{proof}

\begin{remark}
Let $U$ be an object in a symmetric tensor category $\cC$ over a field of characteristic $2$ with trivial Frobenius
twist. Then $\id +c_{U,U}=\iota_U\circ\rho_U$ is the canonical decomposition into an epimorphism followed
by a monomorphism, so $\Gamma^2(U)$ is the image of $\id +c_{U,U}$, and we have
$\bigl(\Gamma^2(U),\rho_U)=\coker \iota_U$. As a consequence, if $\Bup\colon U\otimes U\rightarrow \Buno$ 
is an alternating bilinear form on $U$ (i.e., $\Bup\circ\iota_U=0$), then $\Bup =\Qup\circ\rho_U$ for a 
unique quadratic form $\Qup\colon\Gamma^2(U)\rightarrow\Buno$.
\end{remark}

\begin{remark}\label{re:B_Kannan}
In \cite[Definition 2.5]{Kannanetal}, the associated bilinear form $\Bup$ attached to a quadratic form $\Qup$ on
the object $U$ in a symmetric tensor category $\cC$ over a field of characteristic $2$ is defined as the composition
\[
U\otimes U\xrightarrow{\pi_U} S^2(U)\xrightarrow{q} I\xrightarrow{i}\Gamma^2(U)\xrightarrow{\Qup} \Buno.
\]
As shown in the proof above, this is equivalent to our Definition \ref{df:QB}.
\end{remark}

%-----------------------
\bigskip

\section{Composition algebras in symmetric tensor categories}\label{se:composition}

The goal of this section is the definition of composition algebras over arbitrary symmetric tensor categories.

Actually, over fields of characteristic different from $2$, these have been defined in \cite{DES} using just
the polar form of the norm, but this does not work in general.

\smallskip

Given an object $U$ in a symmetric tensor category $\cC$ over a field $\FF$, consider, as before, the kernel 
$(\Gamma^2(U),\iota_U)$
of the morphism $\id-c_{U,U}\colon U\otimes U\rightarrow U\otimes U$. 

Assume now that $U$ is endowed with a multiplication, that is, with a morphism
\[
\mu\colon U\otimes U\rightarrow U.
\]
This multiplication induces a multiplication $\hat \mu$ in $U\otimes U$ as the composition (as usual, care will not
be taken about the association of parentheses, or about the unit isomorphisms)
\begin{equation}\label{eq:hatmu}
(U\otimes U)\otimes (U\otimes U)\xrightarrow{\id_U\otimes c_{U,U}\otimes \id_U} U\otimes U\otimes U\otimes U
\xrightarrow{\mu\otimes \mu} U\otimes U,
\end{equation}
or, in graphical terms (top-down):
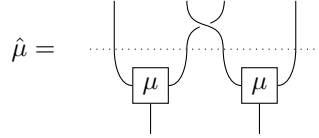
\[
\hat\mu=\quad
\begin{tikzpicture}[scale=0.8,baseline=11ex]
% c23

\draw (-0.6,3) -- (-0.6,2.2);
\draw (2.4,3) -- (2.4,2.2);
  \draw (1.22,3) .. controls (1.2,2.4) and (0.6,2.8) .. (0.6,2.2);
 \draw[line width=4pt, white, smooth] (0.9,2.8) -- (0.9,2.4);
 \draw (0.6,3) .. controls (0.6,2.4) and (1.2,2.8) .. (1.2,2.2);

\draw[thin, dotted] (-1,2.2) -- (2.8,2.2);

% multiplication on the left
\node[draw, inner sep=1.6pt, shape=regular polygon, regular polygon sides=4] (node3) at (0,1.6) {$\mu$};
\draw (0,0.8) -- (node3.south);

  \draw (-0.6,2.2) .. controls (-0.6,1.9) and (-0.5,1.6) .. (node3.west);
  \draw (0.6,2.2) .. controls (0.6,1.9) and (0.5,1.6) .. (node3.east);

%multiplication on the right
\node[draw, inner sep=1.6pt, shape=regular polygon, regular polygon sides=4] (node4) at (1.8,1.6) {$\mu$};
\draw (1.8,0.8) -- (node4.south);

  \draw (1.2,2.2) .. controls (1.2,1.9) and (1.3,1.6) .. (node4.west);
  \draw (2.4,2.2) .. controls (2.4,1.9) and (2.3,1.6) .. (node4.east);

\end{tikzpicture}
\]

This multiplication $\hat\mu$ induces a multiplication in $\Gamma^2(U)$.

\begin{theorem}\label{th:mutilde}
There exists a unique morphism $\tilde\mu\colon \Gamma^2(U)\otimes\Gamma^2(U)\rightarrow \Gamma^2(U)$ such 
that the following diagram
\[
\begin{tikzcd}
\Gamma^2(U)\otimes\Gamma^2(U) \arrow[d,"\iota_U\otimes\iota_U"'] \arrow[r, "\tilde\mu"] 
& \Gamma^2(U) \arrow[d, "\iota_U"] \\
U\otimes U\otimes U\otimes U \arrow[r, "\hat\mu"] & U\otimes U
\end{tikzcd}
\]
is commutative.
\end{theorem}

\begin{proof}
This is better done graphically. The morphism $c_{U,U}\circ\hat\mu\circ(\iota_U\otimes\iota_U)$ corresponds to the diagram
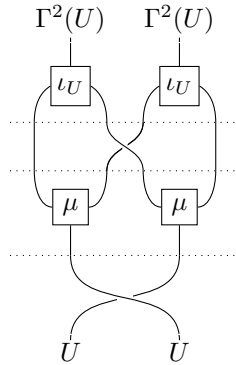
\[
\begin{tikzpicture}[scale=0.8,baseline=-9ex]
  % left iota_U 
  \node[draw, inner sep=.1pt, shape=regular polygon, regular polygon sides=4] (node1) at (0,0) {$\iota_U$};
  \draw (0,.8) -- (node1.north);

  \draw (node1.west) .. controls (-0.5,-0.1) and (-0.6,-0.3) .. (-0.6,-0.6);
  \draw (node1.east) .. controls (0.5,-0.1) and (0.6,-0.3) .. (0.6,-0.6);
  % right iota_u 
\node[draw, inner sep=.1pt, shape=regular polygon, regular polygon sides=4] (node2) at (1.8,0) {$\iota_U$};
  \draw (1.8,.8) -- (node2.north);

  \draw (node2.west) .. controls (1.3,-0.1) and (1.2,-0.3) .. (1.2,-0.6);
  \draw (node2.east) .. controls (2.3,-0.1) and (2.4,-0.3) .. (2.4,-0.6);
  
  \draw[thin, dotted] (-1,-0.6) -- (2.8,-0.6);
 % braid in 23
  \draw (-0.6,-0.6) -- (-0.6,-1.4);
  \draw (2.4,-0.6) -- (2.4,-1.4);
  \draw (1.2,-0.6) .. controls (1.2,-1) and (0.6,-1) .. (0.6,-1.4);
\draw (1.2,-0.6) .. controls (1.2,-1) and (0.6,-1) .. (0.6,-1.4);
 \draw[line width=2pt, white, smooth] (0.9,-0.9) -- (0.9,-1.1);
  \draw (0.6,-0.6) .. controls (0.6,-1) and (1.2,-1) .. (1.2,-1.4);

\draw[thin, dotted] (-1,-1.4) -- (2.8,-1.4);

% multiplication on the left
\node[draw, inner sep=1.6pt, shape=regular polygon, regular polygon sides=4] (node3) at (0,-2) {$\mu$};
\draw (0,-2.8) -- (node3.south);

  \draw (-0.6,-1.4) .. controls (-0.6,-1.7) and (-0.5,-2) .. (node3.west);
  \draw (0.6,-1.4) .. controls (0.6,-1.7) and (0.5,-2) .. (node3.east);

%multiplicationon the right
\node[draw, inner sep=1.6pt, shape=regular polygon, regular polygon sides=4] (node4) at (1.8,-2) {$\mu$};
\draw (1.8,-2.8) -- (node4.south);

  \draw (1.2,-1.4) .. controls (1.2,-1.7) and (1.3,-2) .. (node4.west);
  \draw (2.4,-1.4) .. controls (2.4,-1.7) and (2.3,-2) .. (node4.east);

\draw[thin, dotted] (-1,-2.8) -- (2.8,-2.8);

%braid
  \draw (1.8,-2.8) .. controls (1.8,-3.8) and (0,-3.2) .. (0,-4.2);
  \draw[line width=6pt, white, smooth] (0.9,-3.4) -- (0.9,-3.6);
  \draw (0,-2.8) .. controls (0,-3.8) and (1.8,-3.2) .. (1.8,-4.2);

% initial and final nodes

  \node[draw=white] at (0,1.1) {$\Gamma^2(U)$};
\node[draw=white] at (1.8,1.1) {$\Gamma^2(U)$};
\node[draw=white] at (0,-4.4) {$U$};
\node[draw=white] at (1.8,-4.4) {$U$};
\end{tikzpicture}
\]
which, by naturality of the braiding, equals
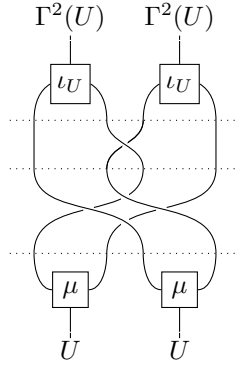
\[
\begin{tikzpicture}[scale=0.8,baseline=-9ex]
  % left iota_U 
  \node[draw, inner sep=.1pt, shape=regular polygon, regular polygon sides=4] (node1) at (0,0) {$\iota_U$};
  \draw (0,.8) -- (node1.north);

  \draw (node1.west) .. controls (-0.5,-0.1) and (-0.6,-0.3) .. (-0.6,-0.6);
  \draw (node1.east) .. controls (0.5,-0.1) and (0.6,-0.3) .. (0.6,-0.6);
  % right iota_u 
\node[draw, inner sep=.1pt, shape=regular polygon, regular polygon sides=4] (node2) at (1.8,0) {$\iota_U$};
  \draw (1.8,.8) -- (node2.north);

  \draw (node2.west) .. controls (1.3,-0.1) and (1.2,-0.3) .. (1.2,-0.6);
  \draw (node2.east) .. controls (2.3,-0.1) and (2.4,-0.3) .. (2.4,-0.6);
  
  \draw[thin, dotted] (-1,-0.6) -- (2.8,-0.6);
 % braid in 23
  \draw (-0.6,-0.6) -- (-0.6,-1.4);
  \draw (2.4,-0.6) -- (2.4,-1.4);
  \draw (1.2,-0.6) .. controls (1.2,-1) and (0.6,-1) .. (0.6,-1.4);
\draw (1.2,-0.6) .. controls (1.2,-1) and (0.6,-1) .. (0.6,-1.4);
 \draw[line width=2pt, white, smooth] (0.9,-0.9) -- (0.9,-1.1);
  \draw (0.6,-0.6) .. controls (0.6,-1) and (1.2,-1) .. (1.2,-1.4);

\draw[thin, dotted] (-1,-1.4) -- (2.8,-1.4);

%braid
  \draw (1.2,-1.4) .. controls (1.2,-2.3) and (-0.6,-1.8) .. (-0.6,-2.8);
  \draw (2.4,-1.4) .. controls (2.4,-2.3) and (0.6,-1.8) .. (0.6,-2.8);
 \draw[line width=3pt, white, smooth] (0.9,-1.7) -- (0.9,-2.5);
 \draw[line width=4pt, white, smooth] (1.5,-1.9) -- (1.5,-2.2);
\draw[line width=4pt, white, smooth] (0.3,-1.9) -- (0.3,-2.2);
  \draw (-0.6,-1.4) .. controls (-0.6,-2.3) and (1.2,-1.8) .. (1.2,-2.8);
  \draw (0.6,-1.4) .. controls (0.6,-2.3) and (2.4,-1.8) .. (2.4,-2.8);

\draw[thin, dotted] (-1,-2.8) -- (2.8,-2.8);

% multiplication on the left
\node[draw, inner sep=1.6pt, shape=regular polygon, regular polygon sides=4] (node3) at (0,-3.4) {$\mu$};
\draw (0,-4.2) -- (node3.south);

  \draw (-0.6,-2.8) .. controls (-0.6,-3.1) and (-0.5,-3.4) .. (node3.west);
  \draw (0.6,-2.8) .. controls (0.6,-3.1) and (0.5,-3.4) .. (node3.east);

%multiplication on the right
\node[draw, inner sep=1.6pt, shape=regular polygon, regular polygon sides=4] (node4) at (1.8,-3.4) {$\mu$};
\draw (1.8,-4.2) -- (node4.south);

  \draw (1.2,-2.8) .. controls (1.2,-3.1) and (1.3,-3.4) .. (node4.west);
  \draw (2.4,-2.8) .. controls (2.4,-3.1) and (2.3,-3.4) .. (node4.east);

% initial and final nodes

  \node[draw=white] at (0,1.1) {$\Gamma^2(U)$};
\node[draw=white] at (1.8,1.1) {$\Gamma^2(U)$};
\node[draw=white] at (0,-4.4) {$U$};
\node[draw=white] at (1.8,-4.4) {$U$};

\end{tikzpicture}
\]
But $c_{U,U}^2=\id$, so this becomes
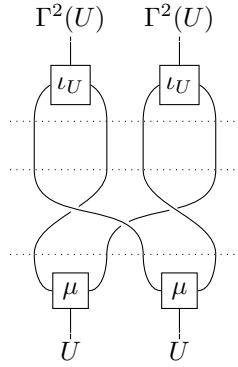
\[
\begin{tikzpicture}[scale=0.8,baseline=-9ex]
  % left iota_U 
  \node[draw, inner sep=.1pt, shape=regular polygon, regular polygon sides=4] (node1) at (0,0) {$\iota_U$};
  \draw (0,.8) -- (node1.north);

  \draw (node1.west) .. controls (-0.5,-0.1) and (-0.6,-0.3) .. (-0.6,-0.6);
  \draw (node1.east) .. controls (0.5,-0.1) and (0.6,-0.3) .. (0.6,-0.6);
  % right iota_u 
\node[draw, inner sep=.1pt, shape=regular polygon, regular polygon sides=4] (node2) at (1.8,0) {$\iota_U$};
  \draw (1.8,.8) -- (node2.north);

  \draw (node2.west) .. controls (1.3,-0.1) and (1.2,-0.3) .. (1.2,-0.6);
  \draw (node2.east) .. controls (2.3,-0.1) and (2.4,-0.3) .. (2.4,-0.6);
  
  \draw[thin, dotted] (-1,-0.6) -- (2.8,-0.6);
 % braid in 23
  \draw (-0.6,-0.6) -- (-0.6,-1.4);
  \draw (0.6,-0.6) -- (0.6,-1.4);
  \draw (1.2,-0.6) -- (1.2,-1.4);
  \draw (2.4,-0.6) -- (2.4,-1.4);

\draw[thin, dotted] (-1,-1.4) -- (2.8,-1.4);

%braid
  \draw (0.6,-1.4) .. controls (0.6,-2) and (-0.6,-2.2) .. (-0.6,-2.8);
  \draw (2.4,-1.4) .. controls (2.4,-2.4) and (0.6,-1.8) .. (0.6,-2.8);
 \draw[line width=3pt, white, smooth] (0.07,-1.9) -- (0.07,-2.2);
 \draw[line width=3pt, white, smooth] (1.7,-1.85) -- (1.7,-2.22);
\draw[line width=3pt, white, smooth] (0.9,-2.1) -- (0.9,-2.4);
  \draw (-0.6,-1.4) .. controls (-0.6,-2.4) and (1.2,-1.8) .. (1.2,-2.8);
  \draw (1.2,-1.4) .. controls (1.2,-2) and (2.4,-2.2) .. (2.4,-2.8);

\draw[thin, dotted] (-1,-2.8) -- (2.8,-2.8);

% multiplication on the left
\node[draw, inner sep=1.6pt, shape=regular polygon, regular polygon sides=4] (node3) at (0,-3.4) {$\mu$};
\draw (0,-4.2) -- (node3.south);

  \draw (-0.6,-2.8) .. controls (-0.6,-3.1) and (-0.5,-3.4) .. (node3.west);
  \draw (0.6,-2.8) .. controls (0.6,-3.1) and (0.5,-3.4) .. (node3.east);

%multiplication on the right
\node[draw, inner sep=1.6pt, shape=regular polygon, regular polygon sides=4] (node4) at (1.8,-3.4) {$\mu$};
\draw (1.8,-4.2) -- (node4.south);

  \draw (1.2,-2.8) .. controls (1.2,-3.1) and (1.3,-3.4) .. (node4.west);
  \draw (2.4,-2.8) .. controls (2.4,-3.1) and (2.3,-3.4) .. (node4.east);

% initial and final nodes

  \node[draw=white] at (0,1.1) {$\Gamma^2(U)$};
\node[draw=white] at (1.8,1.1) {$\Gamma^2(U)$};
\node[draw=white] at (0,-4.4) {$U$};
\node[draw=white] at (1.8,-4.4) {$U$};

\end{tikzpicture}
\]
and since $c_{U,U}\circ\iota_U=\iota_U$, the above diagram equals
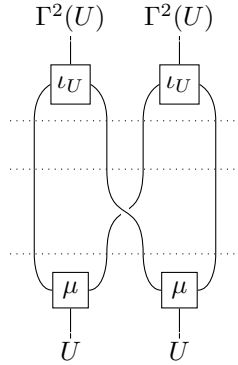
\[
\begin{tikzpicture}[scale=0.8,baseline=-9ex]
  % left iota_U 
  \node[draw, inner sep=.1pt, shape=regular polygon, regular polygon sides=4] (node1) at (0,0) {$\iota_U$};
  \draw (0,.8) -- (node1.north);

  \draw (node1.west) .. controls (-0.5,-0.1) and (-0.6,-0.3) .. (-0.6,-0.6);
  \draw (node1.east) .. controls (0.5,-0.1) and (0.6,-0.3) .. (0.6,-0.6);
  % right iota_u 
\node[draw, inner sep=.1pt, shape=regular polygon, regular polygon sides=4] (node2) at (1.8,0) {$\iota_U$};
  \draw (1.8,.8) -- (node2.north);

  \draw (node2.west) .. controls (1.3,-0.1) and (1.2,-0.3) .. (1.2,-0.6);
  \draw (node2.east) .. controls (2.3,-0.1) and (2.4,-0.3) .. (2.4,-0.6);
  
  \draw[thin, dotted] (-1,-0.6) -- (2.8,-0.6);
 % braid in 23
  \draw (-0.6,-0.6) -- (-0.6,-1.4);
  \draw (0.6,-0.6) -- (0.6,-1.4);
  \draw (1.2,-0.6) -- (1.2,-1.4);
  \draw (2.4,-0.6) -- (2.4,-1.4);

\draw[thin, dotted] (-1,-1.4) -- (2.8,-1.4);

%braid
  \draw (-0.6,-1.4) -- (-0.6,-2.8);
  \draw (2.4,-1.4) -- (2.4,-2.8);

  \draw (1.2,-1.4) .. controls (1.2,-2.4) and (0.6,-1.8) .. (0.6,-2.8);

\draw[line width=3pt, white, smooth] (0.9,-2) -- (0.9,-2.2);
  \draw (0.6,-1.4) .. controls (0.6,-2.4) and (1.2,-1.8) .. (1.2,-2.8);

\draw[thin, dotted] (-1,-2.8) -- (2.8,-2.8);

% multiplication on the left
\node[draw, inner sep=1.6pt, shape=regular polygon, regular polygon sides=4] (node3) at (0,-3.4) {$\mu$};
\draw (0,-4.2) -- (node3.south);

  \draw (-0.6,-2.8) .. controls (-0.6,-3.1) and (-0.5,-3.4) .. (node3.west);
  \draw (0.6,-2.8) .. controls (0.6,-3.1) and (0.5,-3.4) .. (node3.east);

%multiplication on the right
\node[draw, inner sep=1.6pt, shape=regular polygon, regular polygon sides=4] (node4) at (1.8,-3.4) {$\mu$};
\draw (1.8,-4.2) -- (node4.south);

  \draw (1.2,-2.8) .. controls (1.2,-3.1) and (1.3,-3.4) .. (node4.west);
  \draw (2.4,-2.8) .. controls (2.4,-3.1) and (2.3,-3.4) .. (node4.east);

% initial and final nodes

  \node[draw=white] at (0,1.1) {$\Gamma^2(U)$};
\node[draw=white] at (1.8,1.1) {$\Gamma^2(U)$};
\node[draw=white] at (0,-4.4) {$U$};
\node[draw=white] at (1.8,-4.4) {$U$};

\end{tikzpicture}
\]
which corresponds to $\hat\mu\circ(\iota_U\otimes\iota_U)$. Hence we get 
$c_{U,U}\circ\hat\mu\circ(\iota_U\otimes\iota_U)=\hat\mu\circ(\iota_U\otimes\iota_U)$, that is:
\[
(\id -c_{U,U})\circ\hat\mu\circ(\iota_U\otimes\iota_U)=0.
\]
Now it is enough to use that $(\Gamma^2(U),\iota_U)$ is the kernel of $\id -c_{U,U}$.
\end{proof}

In other words, the multiplication $\mu$ in $U$ induces naturally a multiplication $\hat\mu$ in $U\otimes U$, that
restricts to a multiplication $\tilde\mu$ on $\Gamma^2(U)$.

\bigskip

We are now ready to define a composition algebra. Recall (Definition \ref{df:QB}) that a quadratic form $\Qup$ on $U$ is a morphism
$\Qup\colon \Gamma^2(U)\rightarrow \Buno$. Its associated (symmetric) bilinear form $\Bup$ is the morphism 
obtained by composition
\[
U\otimes U\xrightarrow{\rho_U}\Gamma^2(U)\xrightarrow{\ \Qup\ }\Buno,
\]
where $\rho_U$ is the only morphism $U\otimes U\rightarrow \Gamma^2(U)$ such that $\iota_U\circ \rho_U=
\id+c_{U,U}$. The quadratic form $\Qup$ is said to be \emph{nondegenerate} if so is $\Bup$. 

\begin{definition}\label{df:compoSTC}
Let $\cC$ be a symmetric tensor category.
\begin{itemize}
\item Let $U$ be an object in $\cC$ endowed with a quadratic form 
$\Qup\colon U\rightarrow\Buno$ and a multiplication $\mu\colon U\otimes U\rightarrow U$. Let $\tilde \mu$ be
the induced multiplication in $\Gamma^2(U)$ as in Theorem \ref{th:mutilde}. Then the quadratic form $\Qup$ is 
said to be \emph{multiplicative} (relative to $\mu$) if 
\begin{equation}\label{eq:multiplicative}
\Qup\circ\tilde\mu=\Qup\otimes \Qup\colon \Gamma^2(U)\otimes\Gamma^2(U)\rightarrow \Buno.
\end{equation}
(In other words, $\Qup\colon \Gamma^2(U)\rightarrow \Buno$ is a morphism of algebras, where $\Buno$ has
the trivial multiplication.)

\item A \emph{composition algebra} in $\cC$ is a triple $(U,\mu,\Qup)$, where $U$ is an object in $\cC$, $\mu$
a multiplication in $U$, and $\Qup$ a nondegenerate multiplicative quadratic form on $U$ relative to $\mu$.

\item A \emph{unital composition algebra} in $\cC$ is a quadruple $(U,\mu,\Qup,1_U)$, where $(U,\mu,\Qup)$ is
a composition algebra and $1_U\colon \Buno\rightarrow U$ is a morphism in $\cC$ such that
$\mu\circ(1_U\otimes \id_U)=\id_U=\mu\circ(\id_U\otimes 1_U)$.
\end{itemize}
\end{definition}

It is useful to deal with the \emph{linearizations} of the multiplicative property $\Qup\circ\tilde\mu=\Qup\otimes \Qup$
in \eqref{eq:multiplicative}.

\begin{lemma}\label{le:firstlinearization}
Let $(U,\mu,\Qup)$ be a composition algebra in a symmetric tensor category $\cC$, then we have the equality
\[
\tilde\mu\circ (\id\otimes \rho_U)=\rho_U\circ\hat\mu\circ(\iota_U\otimes\id)
\]
of morphisms $\Gamma^2(U)\otimes U\otimes U\rightarrow \Gamma^2(U)$.

In the same vein, we too have the equality
\[
\tilde\mu\circ (\rho_U\otimes \id)=\rho_U\circ\hat\mu\circ(\id\otimes\iota_U).
\]
\end{lemma}
\begin{proof}
As $\iota_U$ is a monomorphism, it is enough to prove
\begin{equation}\label{eq:iumu}
\iota_U\circ\tilde\mu\circ (\id\otimes \rho_U)=\iota_U\circ\rho_U\circ\hat\mu\circ(\iota_U\otimes\id).
\end{equation}
But $\iota_U\circ\rho_U=\id+c$ (see \eqref{eq:rhoU}), where $c=c_{U,U}$, so the right hand side equals
\[
(\id + c)\circ\hat\mu\circ(\iota_U\otimes \id).
\] 
(Note that we are using the symbol $\id$ for the identity morphism
on various objects, this should create no confusion.)

On the other hand, for the left hand side in \eqref{eq:iumu}, we have
\[
\begin{split}
\iota_U\circ\tilde\mu\circ (\id\otimes \rho_U)&=
   \hat\mu\circ(\iota_U\otimes\iota_U)\circ(\id\otimes\rho_U)\quad \text{(Theorem 2.1)}\\
   &=\hat\mu\circ(\iota_U\otimes(\id +c))=\hat\mu\circ(\id\otimes(\id+c))\circ(\iota_U\otimes \id),
\end{split}
\]
and, hence, it is enough to check the equality
\[
\hat\mu\circ(\id\otimes c)\circ(\iota_U\otimes\id)=c\circ\hat\mu\circ(\iota_U\otimes\id),
\]
and, since $c\circ\iota_U=\iota_U$, this amounts to the equality
\[
\hat\mu\circ(\id\otimes c)\circ(\iota_U\otimes\id)=c\circ\hat\mu\circ(c\otimes\id)\circ(\iota_U\otimes\id).
\]
Therefore, it is enough to prove the equality
\[
\hat\mu\circ(\id\otimes c)=c\circ\hat\mu\circ(c\otimes\id),
\]
or, composing on the right by $c\otimes\id$, the more symmetric equality
\begin{equation}\label{eq:hatmucs}
\hat\mu\circ(c\otimes c)=c\circ\hat\mu.
\end{equation}
But this is easy:
\[
c\circ\hat\mu=\
\begin{tikzpicture}[scale=0.8,baseline=7ex]
% c23

\draw (-0.6,3) -- (-0.6,2.2);
\draw (2.4,3) -- (2.4,2.2);
  \draw (1.22,3) .. controls (1.2,2.4) and (0.6,2.8) .. (0.6,2.2);
 \draw[line width=4pt, white, smooth] (0.9,2.8) -- (0.9,2.4);
 \draw (0.6,3) .. controls (0.6,2.4) and (1.2,2.8) .. (1.2,2.2);

\draw[thin, dotted] (-1,2.2) -- (2.8,2.2);

% multiplication on the left
\node[draw, inner sep=1.6pt, shape=regular polygon, regular polygon sides=4] (node3) at (0,1.6) {$\mu$};
\draw (0,0.8) -- (node3.south);

  \draw (-0.6,2.2) .. controls (-0.6,1.9) and (-0.5,1.6) .. (node3.west);
  \draw (0.6,2.2) .. controls (0.6,1.9) and (0.5,1.6) .. (node3.east);

%multiplication on the right
\node[draw, inner sep=1.6pt, shape=regular polygon, regular polygon sides=4] (node4) at (1.8,1.6) {$\mu$};
\draw (1.8,0.8) -- (node4.south);

  \draw (1.2,2.2) .. controls (1.2,1.9) and (1.3,1.6) .. (node4.west);
  \draw (2.4,2.2) .. controls (2.4,1.9) and (2.3,1.6) .. (node4.east);

\draw[thin, dotted] (-1,0.8) -- (2.8, 0.8);

% braiding

  \draw (1.8,0.8) .. controls (1.8,-0.2) and (0,0.6) .. (0,-0.4);
\draw[line width=8pt, white, smooth] (0.9,0.3) -- (0.9,0.1);
  \draw (0,0.8) .. controls (0,-0.2) and (1.8,0.6) .. (1.8,-0.4);
\end{tikzpicture}
\quad =\quad
\begin{tikzpicture}[scale=0.8,baseline=7ex]
% c23

\draw (-0.6,3) -- (-0.6,2.2);
\draw (2.4,3) -- (2.4,2.2);
  \draw (1.22,3) .. controls (1.2,2.4) and (0.6,2.8) .. (0.6,2.2);
 \draw[line width=4pt, white, smooth] (0.9,2.8) -- (0.9,2.4);
 \draw (0.6,3) .. controls (0.6,2.4) and (1.2,2.8) .. (1.2,2.2);

\draw[thin, dotted] (-1,2.2) -- (2.8,2.2);

%braiding

  \draw (2.4,2.2) .. controls (2.4,1.2) and (0.6,1.8) .. (0.6,0.8);
  \draw (1.2,2.2) .. controls (1.2,1.2) and (-0.6,1.8) .. (-0.6,0.8);
\draw[line width=3pt, white, smooth] (0.9,1.8) -- (0.9,1.2);
\draw[line width=4pt, white, smooth] (0.3,1.6) -- (0.3,1.4);
\draw[line width=4pt, white, smooth] (1.5,1.6) -- (1.5,1.4);
  \draw (-0.6,2.2) .. controls (-0.6,1.2) and (1.2,1.8) .. (1.2,0.8);
  \draw (0.6,2.2) .. controls (0.6,1.2) and (2.4,1.8) .. (2.4,0.8);

\draw[thin, dotted] (-1,0.8) -- (2.8, 0.8);

% multiplication on the left
\node[draw, inner sep=1.6pt, shape=regular polygon, regular polygon sides=4] (node3) at (0,0.2) {$\mu$};
\draw (0,-0.6) -- (node3.south);

  \draw (-0.6,0.8) .. controls (-0.6,0.5) and (-0.5,0.2) .. (node3.west);
  \draw (0.6,0.8) .. controls (0.6,0.5) and (0.5,0.2) .. (node3.east);

%multiplication on the right
\node[draw, inner sep=1.6pt, shape=regular polygon, regular polygon sides=4] (node4) at (1.8,0.2) {$\mu$};
\draw (1.8,-0.6) -- (node4.south);

  \draw (1.2,0.8) .. controls (1.2,0.5) and (1.3,0.2) .. (node4.west);
  \draw (2.4,0.8) .. controls (2.4,0.5) and (2.3,0.2) .. (node4.east);

\end{tikzpicture}
\ =\
\begin{tikzpicture}[scale=0.8,baseline=7ex]
% identity

\draw (-0.6,3) -- (-0.6,2.2);
\draw (0.6,3) -- (0.6,2.2);
\draw (1.2,3) -- (1.2,2.2);
\draw (2.4,3) -- (2.4,2.2);

\draw[thin, dotted] (-1,2.2) -- (2.8,2.2);

%braiding

  \draw (2.4,2.2) .. controls (2.4,1.4) and (0.6,1.3) .. (0.6,0.8);
  \draw (0.6,2.2) .. controls (0.6,1.2) and (-0.6,1.8) .. (-0.6,0.8);
\draw[line width=4pt, white, smooth] (0.9,1.25) -- (0.9,1);
\draw[line width=4pt, white, smooth] (0,1.6) -- (0,1.4);
\draw[line width=4pt, white, smooth] (1.75,1.6) -- (1.75,1.4);
  \draw (-0.6,2.2) .. controls (-0.6,1.4) and (1.2,1.3) .. (1.2,0.8);
\draw (1.2,2.2) .. controls (1.2,1.2) and (2.4,1.8) .. (2.4,0.8);

\draw[thin, dotted] (-1,0.8) -- (2.8, 0.8);

% multiplication on the left
\node[draw, inner sep=1.6pt, shape=regular polygon, regular polygon sides=4] (node3) at (0,0.2) {$\mu$};
\draw (0,-0.6) -- (node3.south);

  \draw (-0.6,0.8) .. controls (-0.6,0.5) and (-0.5,0.2) .. (node3.west);
  \draw (0.6,0.8) .. controls (0.6,0.5) and (0.5,0.2) .. (node3.east);

%multiplication on the right
\node[draw, inner sep=1.6pt, shape=regular polygon, regular polygon sides=4] (node4) at (1.8,0.2) {$\mu$};
\draw (1.8,-0.6) -- (node4.south);

  \draw (1.2,0.8) .. controls (1.2,0.5) and (1.3,0.2) .. (node4.west);
  \draw (2.4,0.8) .. controls (2.4,0.5) and (2.3,0.2) .. (node4.east);

\end{tikzpicture}
\]
where the second equality is given by the naturality of the brading, and the third equality follows from $c^2=\id$. 
The last diagram corresponds to $\hat\mu\circ(c\otimes c)$, as required.
\end{proof}

We are ready for the linearizations of the multiplicative property.

\begin{proposition}\label{pr:linearizations}
Let $(U,\mu)$ be an algebra in a symmetric tensor category $\cC$, and let $\Qup$ be a multiplicative 
quadratic form on $U$
(i.e.,  Equation \eqref{eq:multiplicative} holds). Let $\Bup$ be the symmetric
bilinear form associated to $\Qup$ ($\Bup=\Qup\circ\rho_U$). Then the following equalities hold:
\begin{itemize}
\item $\Qup\otimes \Bup=\Bup\circ\hat\mu\circ(\iota_U\otimes \id)$ \emph{(right partial linearization)}.
\item $\Bup\otimes \Qup=\Bup\circ\hat\mu\circ(\id\otimes \iota_U)$ \emph{(left partial linearization)}.
\item $\Bup\otimes \Bup=\Bup\circ\hat\mu\circ((\id +c)\otimes\id)=\Bup\circ\hat\mu\circ(\id\otimes(\id+c))$
\emph{(full linearization)}.
\end{itemize}
\end{proposition}

\begin{remark}
Note that in the category of vector spaces, the right  and left partial linearizations correspond to the usual formulas
\[
\Qup(x)\Bup(y,z)=\Bup(xy,xz),\quad \Bup(y,z)\Qup(x)=\Bup(yx,zx),\quad\forall x,y,z,
\]
while the full linearization corresponds to the formula
\[
\Bup(xz,yt)+\Bup(xt,yz)=\Bup(x,y)\Bup(z,t)\quad\forall x,y,z,t.
\]
\end{remark}

\begin{proof}[Proof of Proposition \ref{pr:linearizations}]
For the right linearization note that
\[
\begin{split}
\Qup\otimes \Bup&=(\Qup\otimes \Qup)\circ(\id\otimes\rho_U)\quad \text{(definition of $\Bup$)}\\
 &=\Qup\circ\tilde\mu\circ(\id\otimes\rho_U)\quad \text{by \eqref{eq:multiplicative}}\\
 &=\Qup\circ\rho_U\circ\hat\mu\circ(\iota_U\otimes\id)\quad\text{(Lemma \ref{le:firstlinearization})}\\
 &=\Bup\circ\hat\mu\circ(\iota_U\otimes\id).
\end{split}
\]
The left linearization is proven along the same lines.

For the full linearization we proceed as follows:
\[
\begin{split}
\Bup\otimes \Bup&=(\Qup\otimes \Bup)\circ(\rho_U\otimes \id)\quad\text{(definition of $\Bup$)}\\
 &=\Bup\circ\hat\mu\circ(\iota_U\circ\rho_U\otimes\id)\quad\text{(right linearization)}\\
 &=\Bup\circ\hat\mu\circ((\id + c)\otimes\id),
\end{split}
\]
and, in the same vein, but using the left linearization one gets 
$\Bup\otimes \Bup=\Bup\circ\hat\mu\circ(\id\otimes(\id + c))$, as required.
\end{proof}

\begin{remark} Under the hypotheses of Proposition \ref{pr:linearizations}, note that since $\Bup$ is symmetric
we get
\[
\Bup\circ\hat\mu\circ(c\otimes \id)=\Bup\circ c\circ\hat\mu\circ(c\otimes\id),
\]
and since $c\circ\hat\mu=\hat\mu\circ (c\otimes c)$ by \eqref{eq:hatmucs}, we get
$c\circ\hat\mu\circ(c\otimes\id)=\hat\mu\circ(\id\otimes c)$, and hence we obtain a direct proof of
\[
\Bup\circ\hat\mu\circ(c\otimes \id)=\Bup\circ\hat\mu\circ(\id\otimes c).
\]
\end{remark}

Under some circumstances, the full linearization of the multiplicative property \eqref{eq:multiplicative} 
implies the multiplicative property itself.

\begin{proposition}\label{pr:full_multiplicative}
Let $(U,\mu)$ be an algebra in a symmetric tensor category $\cC$ over a field $\FF$, 
and let $\Qup$ be a quadratic form on $U$
and $\Bup$ the associated symmetric bilinear form. Let $\hat\mu$ and $\tilde\mu$ be the multiplications
induced on $U\otimes U$ and on $\Gamma^2(U)$. If either
\begin{itemize}
\item the characteristic of $\FF$ is not $2$, or
\item the characteristic of $\FF$ is $2$ and the Frobenius twist is trivial: $U^{(1)}=0$,
\end{itemize}
then the full linearization $\Bup\otimes \Bup=\Bup\circ\hat\mu\circ((\id +c)\otimes\id)$ implies the
multiplicative property of $\Qup$: $\Qup\otimes \Qup=\Qup\circ\tilde\mu$.
\end{proposition}
\begin{proof}
\iffalse
If $\Bup\otimes\Bup=\Bup\circ\hat\mu\circ((\id +c)\otimes \id)$, then composing with $\iota_U\otimes\iota_U$ 
on the right we get 
$(\Bup\otimes\Bup)\circ(\iota_U\otimes\iota_U)
=\Bup\circ\hat\mu\circ((\id +c)\otimes \id)\circ(\iota_U\otimes\iota_U)$. But we have
\[
\begin{split}
(\Bup\otimes\Bup)\circ(\iota_U\otimes\iota_U)
   &=(\Qup\otimes\Qup)\circ(\rho_U\otimes\rho_U)\circ(\iota_U\otimes\iota_U)\\
   &=4(\Qup\otimes\Qup)\quad\text{(because $\rho_U\circ\iota_U=2\id_{\Gamma^2(U)}$)},
\end{split}
\]
while
\[
\begin{split}
\Bup\circ\hat\mu\circ((\id +c)\otimes \id)\circ(\iota_U\otimes\iota_U)
 &=2\Bup\circ\hat\mu\circ(\iota_U\otimes\iota_U)\quad\text{(as $c\circ\iota_U=\iota_U$)}\\
 &=2\Bup\circ\iota_U\circ\tilde\mu\quad\text{(Theorem \ref{th:mutilde})}\\
 &=2\Qup\circ\rho_U\circ\iota_U\circ\tilde\mu\\
 &=4\Qup\circ\tilde\mu\quad\text{(as 
 $\rho_U\circ\iota_U=2\id_{\Gamma^2(U)}$)}.
\end{split}
\]
Hence, if the characteristic of $\FF$ is not $2$, the result follows.

Assume now that the characteristic of the ground field is $2$. 
\fi
The full linearization 
$\Bup\otimes \Bup=\Bup\circ\hat\mu\circ((\id +c)\otimes\id)$ can be written as
\[
(\Qup\otimes\Qup)\circ(\rho_U\otimes\rho_U)=\Qup\circ\rho_U\circ\hat\mu\circ((\id + c)\otimes\id).
\]

But Lemma \ref{le:firstlinearization} gives 
\[
\tilde\mu\circ(\rho_U\otimes\rho_U)=\rho_U\circ\hat\mu\circ((\iota_U\circ\rho_U)\otimes\id)
=\rho_U\circ\hat\mu\circ((\id +c)\otimes \id),
\]
and hence the full linearization gives
\[
(\Qup\otimes\Qup)\circ(\rho_U\otimes\rho_U)=\Qup\circ\tilde\mu\circ(\rho_U\otimes\rho_U).
\]
If either the characteristic of $\FF$ is not $2$ or the characteristic is $2$ but the Frobenius twist $U^{(1)}$ is trivial,
 then $\rho_U$ is an epimorphism (Propositions \ref{pr:rhoUnot2} and \ref{pr:rhoU2}), and so is
$\rho_U\otimes\rho_U$, so the above equation gives $\Qup\otimes \Qup=\Qup\circ\tilde\mu$, as required.
\end{proof}

%-----------------------------

\bigskip

\section{Examples}\label{se:examples}
%Categories of vector spaces and vector superspaces

In this section, the unital composition algebras in the symmetric tensor categories of finite-dimensional
vector spaces $\vect$ or superspaces $\sVec$ will be revisited.

\subsection{Unital composition algebras in \texorpdfstring{$\vect$}{Vec}}\null\quad

In the symmetric tensor category $\vect$ over a field $\FF$, the braiding is given by swapping the components: 
$u\otimes v\mapsto v\otimes u$. Hence, given an object $U$ in $\vect$, $\Gamma^2(U)$ is the subspace
of the symmetric elements
in the usual tensor product $U\otimes_\FF U$. A linear map $\Qup\colon \Gamma^2(U)\rightarrow \Buno=\FF$
is determined by the values $\Qup(u\otimes u)$, with $u\in U$. For $u\in U$, we will write 
$\Qup(u)\bydef \Qup(u\otimes u)$. This gives a map, denoted by the same symbol, $\Qup\colon U\rightarrow \FF$
which is a quadratic form in the usual sense. Besides, if $\Bup\colon U\otimes U\rightarrow \Buno=\FF$ is
the associated bilinear form (Definition \ref{df:QB}), write $\Bup(u,v)\bydef \Bup(u\otimes v)$. This gives
the \emph{polar form} of $\Qup$:
\begin{multline*}
\Bup(u,v)=\Bup(u\otimes v)=\Qup\circ\rho_U(u\otimes v)=\Qup(u\otimes v+v\otimes u)\\
=
\Qup\bigl((u+v)\otimes(u+v)-u\otimes u-v\otimes v\bigr)=\Qup(u+v)-\Qup(u)-\Qup(v),
\end{multline*}
for any $u,v\in U$.

Hence, the unital composition algebras in $\vect$ according to Definition \ref{df:compoSTC} are the unital 
composition algebras (or Hurwitz algebras) in the usual sense (as, for example, in \cite[\S 33.C]{KMRT}), with
just one subtle difference.

In absence of a notion of regular quadratic form as in \cite[p.~xix]{KMRT}, where the quadratic form 
$q\colon U\rightarrow\FF$ is said to be regular if either its polar form is nondegenerate or if the radical
of the polar form has dimension $1$ and $q$ is not trivial on it, our Definition \ref{df:compoSTC} imposes
the associated bilinear form $\Bup$ to be nondegenerate. This is the condition imposed in the
definition of unital composition algebra in \cite[Chapter 2]{ZSSS}. As a consequence, the classification of the unital
composition algebras in $\vect$ is given by \cite[Chapter 2, Theorem 1]{ZSSS}:

\begin{theorem}\label{th:composition_vec}
Any unital composition algebra in the symmetric tensor category $\vect$ over a field $\FF$ is, 
up to isomorphism, one of the following:
\begin{itemize}
\item The ground field $\FF$ if the characteristic is $\neq 2$.
\item An \'etale algebra of dimension $2$ over $\FF$. That is, either $\FF\times\FF$ or a degree two separable
field extension of $\FF$.
\item A quaternion algebra over $\FF$.
\item A Cayley (or octonion) algebra over $\FF$.
\end{itemize}
The quadratic form $\Qup$ is the unique quadratic form satisfying $\Qup(1)=1$ and $x^2-\Bup(x,1)x+\Qup(x)1=0$
for all $x$.
\end{theorem}

\subsection{Unital composition algebras in \texorpdfstring{$\sVec$}{sVec}}\null\quad

The objects of the symmetric tensor category $\sVec$ of vector superspaces over a field $\FF$ are the
$\ZZ_2$-graded vector spaces $U=U\subo\oplus U\subuno$. The braiding is given by 
$u\otimes v\mapsto (-1)^{\lvert u\rvert\lvert v\rvert}v\otimes u$, where $\lvert u\rvert$ is $0$ if $u$ is
an even element ($u\in U\subo$),
and $1$ if $u$ is odd ($u\in U\subuno$). 

Assume first that the characteristic of $\FF$ is not $2$. Then, given an object $U$ in $\sVec$, $\Gamma^2(U)$
is the subspace spanned by the elements
$u\otimes v+v\otimes u$, where at least one of $u$ and $v$ is even, and the elements $u\otimes v-v\otimes u$, 
where both $u$ and $v$ are odd. A morphism $\Qup\colon\Gamma^2(U)\rightarrow \Buno=\FF$ ($\FF$ is
endowed with the trivial grading $\FF=\FF\subo$) is determined by the values $\Qup(u\otimes u)$, for $u\in U\subo$,
and the values $\Bup(u\otimes v)=\Qup(u\otimes v-v\otimes u)$ for $u,v\in U\subuno$.

Writing $\Qup(u)\bydef \Qup(u\otimes u)$ for $u\in U\subo$, and $\Bup(u,v)=\Bup(u\otimes v)$ for
arbitrary $u,v\in U$, $\Qup$ gives a usual quadratic form on $U\subo$, while $\Bup$ gives a supersymmetric
bilinear form on $U$: $\Bup$ is symmetric on $U\subo$, alternating on $U\subuno$, and 
$\Bup(U\subo,U\subuno)=0=\Bup(U\subuno,U\subo)$. Proposition \ref{pr:full_multiplicative} shows that, in order
to check that the quadratic form $\Qup$ is multiplicative, it is enough to use the full linearization.

As a consequence,the unital composition algebras in $\sVec$ according to Definition \ref{df:compoSTC} coincide
with the unital composition superalgebras in \cite[Definition 2.2]{EOsuper}.

\smallskip

However, if the characteristic of the ground field $\FF$ is $2$, then $\sVec$ is just the symmetric tensor category
of the finite-dimensional $\ZZ_2$-graded vector spaces, with the usual braiding: $u\otimes v\mapsto v\otimes u$,
and hence the quadratic forms correspond to usual quadratic forms whose polar forms satisfy 
$\Bup(U\subo,U\subuno)=0$, and the
unital composition algebras in $\sVec$ are just the $\ZZ_2$-graded unital composition algebras
in $\vect$. Since the definition of a quadratic form on $\sVec$ is more restrictive than the definition of
quadratic superform in \cite[Definition 2.1]{EOsuper}, it turns out that the definition of unital composition
algebra in $\sVec$ does not agree in this case with the definition of unital composition superalgebra
in \cite[Definition 2.2]{EOsuper}, which is less restrictive.

But the classification of the unital composition superalgebras in \cite[Theorem 3.1]{EOsuper} implies
that any unital composition superalgebra is a unital composition algebra in $\sVec$. Hence the above cited
result gives the following classification (for details, the reader should consult \cite{EOsuper}):

\begin{theorem}
Any unital composition algebra in the symmetric tensor category $\sVec$ over a field $\FF$ is,
up to isomorphism, one of the following:
\begin{itemize}
\item A unital composition algebra in $\vect$, seen as a unital composition algebra in $\sVec$ with trivial odd part.
\item One of the superalgebras $B(1,2)$ or $B(4,2)$ in characteristic $3$.
\item A quaternion  algebra $\cQ$ over $\FF$, obtained from the Cayley-Dickson doubling process
from an \'etale subalgebra $\cK$ of dimension $2$, with $\cQ\subo=\cK$ and $\cQ\subuno=\cK^\perp$, in characteristic $2$.
\item or a Cayley algebra $\cC$ over $\FF$, obtained from the Cayely-Dickson doubling process from a 
quaternion  subalgebra $\cQ$, with $\cC\subo=\cQ$ and $\cC\subuno=\cQ^\perp$, in characteristic $2$.
\end{itemize}
\end{theorem}

The superalgebras $B(1,2)$ and $B(2,4)$, specific of characteristic $3$, have been obtained in
\cite{DES} through a process of semisimplification from the split Cayley algebra, looking at this latter algebra as
an algebra in the category $\Repe$ of representations of the cyclic group of order $3$, whose semisimplification
is the Verlinde category $\Ver_3$, which is equivalent to $\sVec$.

%-------------------------
\bigskip

\section{Unital composition algebras in \texorpdfstring{$\Ver_4^+$}{Ver4+}}\label{se:ver4}

The goal of this last section is the classification of the unital composition algebras in the symmetric
tensor category $\Ver_4^+$. After deriving some consequences of the composition property in this
category, the classification will be split in two cases, depending on whether the Frobenius twist
is trivial or not. 

\medskip

Let $(U,\mu,\Qup,1_U)$ be a unital composition algebra in the symmetric tensor category $\Ver_4^+$
over a field $\FF$ of characteristic $2$, with multiplication $\mu$
denoted by juxtaposition: $\mu(u\otimes v)=uv$ for all $u,v\in U$,  with associated quadratic form 
$\Qup\colon\Gamma^2(U)\rightarrow\FF$, and with unity $1_U\colon \Buno=\FF\rightarrow U$ . Denote by 
$1\in U$ the image of the unity of $\Buno=\FF$ under $1_U$, so that $1$ is the identity element in the usual sense:
$1u=u=u1$ for any $u\in U$. Therefore, we will write $(U,\mu,\Qup,1)$, where $1$ is understood to be
the identity element. Let $\Bup\colon U\otimes U\rightarrow \FF$ be the associated bilinear form:
\[
\Bup(u\otimes v)\bydef \Qup\bigl(\rho_U(u\otimes v)\bigr)
=\Qup\bigl(u\otimes v+v\otimes u+t(v)\otimes t(u)\bigr)
\]
for all $u,v\in U$. As before (Equation \eqref{eq:rhoU}), $\rho_U\colon U\otimes U\rightarrow \Gamma^2(U)$ 
is the morphism such that
$\iota_U\circ \rho_U=\id +c_{U,U}$, where $\iota_U\colon\Gamma^2(U)\rightarrow U\otimes U$ is the natural
monomorphism. In $\Ver_4^+$ we identify $\Gamma^2(U)$ with a vector subspace of $U\otimes U$, so
that $\iota_U$ is the inclusion.

A word of caution is needed. Any algebra in $\Ver_4^+$ is, in particular, an algebra in $\vect$. A subalgebra 
(respectively, ideal) in $\Ver_4^+$ is a subalgebra (resp., ideal) in $\vect$ invariant under the action of $t$.
Unless otherwise indicated, the words subalgebra and ideal will refer to the category $\Ver_4^+$, although this
will create no confusion, as the ones we are going to consider are trivially invariant under the action of $t$.

For simplicity, we will write
\[
\Bup(u,v)\bydef \Bup(u\otimes v),\quad \Qup(x)\bydef \Qup(x\otimes x)
\]
for all $u,v\in U$ and $x\in \ker t$. Note that the restriction of $\Qup$ to $\ker t$ is a usual quadratic form,
i.e., a quadratic form in $\vect$. Also, $\Bup$ is a usual bilinear form too, and as such it is symmetric too
(\cite[Lemma 3.3]{Kannanetal}), because, for any $u,v\in U$,
\begin{multline*}
\Bup(u,v)=\Bup(u\otimes v)=\Bup\circ c_{U,U}(u\otimes v)=\Bup\bigl(v\otimes u+t(v)\otimes t(u)\bigr)\\
 =\Bup(v\otimes u)+\Bup\bigl(v\otimes t(t(u))\bigr)=\Bup(v\otimes u)+0=\Bup(v,u),
\end{multline*}
due to the fact that, $\Bup\colon U\otimes U\rightarrow \Buno=\FF$ being a morphism in $\Ver_4^+$ forces $0=\Bup\bigl(t(v\otimes u)\bigr)=\Bup\bigl(t(v)\otimes u+v\otimes t(u)\bigr)$. In other words, $\Bup$
satisfies
\begin{equation}\label{eq:Btinvariant}
\Bup\bigl(t(v),u\bigr)=\Bup\bigl(v,t(u)\bigr)
\end{equation}
for all $u,v\in U$. The nondegeneracy of the symmetric bilinear form 
$\Bup$ in $\Ver_4^+$ is equivalent to its nondegeneracy in $\vect$.

As in Theorem \ref{th:mutilde}, we will denote by $\tilde\mu$ the multiplication induced by $\mu$ on
$\Gamma^2(U)$.  

\begin{proposition}\label{pr:t1kertimt}
Let $(U,\mu,1)$ be a unital algebra in $\Ver_4^+$. Then the following assertions hold:
\begin{itemize}
 \item $t(1)=0$.
 \item $\im t$ is an ideal of the subalgebra $\ker t$.
\end{itemize}
\end{proposition}
\begin{proof}
$(U,\mu)$ being an algebra in $\Ver_4^+$, $t$ acts as a derivation and hence $t(1)=t(1^2)=t(1)+t(1)=0$.
Also, it is clear that $\ker t$ is a subalgebra and for any $x\in\ker t$ and $y=t(u)\in\im t$, 
$xy=xt(u)=t(xu)\in\im t$ and, similarly $yx\in \im t$.
\end{proof}

\smallskip

\subsection{Consequences of the composition property}
\ \null

The subspace $\Gamma^2(U)$ of $U\otimes U$ is spanned by the elements $x\otimes x$, for $x\in \ker t$,
and $\rho_U(u\otimes v)=u\otimes v+v\otimes u+t(v)\otimes t(u)$, for $u,v\in U$.

Therefore, the composition property $Q\otimes Q=Q\circ\tilde\mu$ splits into four instances:
\begin{enumerate}
\item 
For $x,y\in\ker t$,
\[
(\Qup\otimes \Qup)\bigl((x\otimes x)\otimes (y\otimes y)\bigr)
     =\Qup\circ\tilde\mu\bigl((x\otimes x)\otimes (y\otimes y)\bigr),
\]
that is, 
\begin{equation}\label{eq:Qxy}
\Qup(xy)=\Qup(x)\Qup(y),
\end{equation}
for all $x,y\in\ker t$.

\item For $x\in \ker t$ and $u,v\in U$,
\[
(\Qup\otimes\Qup)\bigl((x\otimes x)\otimes\rho_U(u\otimes v)\bigr)
 =\Qup\circ\tilde\mu\bigl((x\otimes x)\otimes\rho_U(u\otimes v)\bigr).
\]
The left hand side equals $(\Qup\otimes\Bup)\bigl((x\otimes x)\otimes (u\otimes v)\bigr)=\Qup(x)\Bup(u,v)$,
while the right hand side equals (see the proof of the right partial linearization in Proposition \ref{pr:linearizations}) 
\[
\Bup\circ\hat\mu\bigl((x\otimes x)\otimes (u\otimes v)\bigr)=\Bup(xu\otimes xv)=\Bup(xu,xv).
\]
Hence we get
\begin{equation}\label{eq:QBxuv}
\Bup(xu,xv)=\Qup(x)\Bup(u,v),
\end{equation}
for all $x\in\ker t$ and $u,v\in U$.

\item In the same vein, for $x\in\ker t$ and $u,v\in U$,
\[
(\Qup\otimes\Qup)\bigl(\rho_U(u\otimes v)\otimes (x\otimes x)\bigr)
 =\Qup\circ\tilde\mu\bigl(\rho_U(u\otimes v)\otimes (x\otimes x)\bigr).
\]
which is equivalent to
\begin{equation}\label{eq:BQuvx}
\Bup(ux,vx)=\Qup(x)\Bup(u,v),
\end{equation}
for all $x\in\ker t$ and $u,v\in U$.

\item For $x,y,z,w\in U$,
\[
(\Qup\otimes\Qup)\bigl(\rho_U(x\otimes y)\otimes \rho_U(z\otimes w)\bigr)
 =\Qup\circ\tilde\mu\bigl(\rho_U(x\otimes y)\otimes \rho_U(z\otimes w)\bigr).
\]
The left hand side equals $(\Bup\otimes\Bup)\bigl((x\otimes y)\otimes (z\otimes w)\bigr)=\Bup(x,y)\Bup(z,w)$.
The right hand side equals (see the proof of the full linearization in Proposition \ref{pr:linearizations})
\[
(\Bup\circ\hat\mu)\circ\bigl((\id+c_{U,U})\otimes \id\bigr)\bigl((x\otimes y)\otimes (z\otimes w)\bigr).
\]
Straightforward computations give
\[
\begin{split}
(x&\otimes y)\otimes (z\otimes w)\\
 &\xrightarrow{(\id+c_{U,U})\otimes \id}x\otimes y\otimes z\otimes w+
                                              y\otimes x \otimes z\otimes w+ t(y)\otimes t(x)\otimes z\otimes w\\
 &\xrightarrow{\null\ c_{23}\ \null} x\otimes z\otimes y\otimes w + x\otimes t(z)\otimes t(y)\otimes w +
        y\otimes z\otimes x\otimes w\\ 
 &\qquad\qquad + y\otimes t(z)\otimes t(x)\otimes w + t(y)\otimes z\otimes t(x)\otimes w\\
 &\xrightarrow{\,\mu\otimes\mu\,} xz\otimes yw +xt(z)\otimes t(y)w + yz\otimes xw +yt(z)\otimes t(x)w
        +t(y)z\otimes t(x)w\\
&\xrightarrow{\null\ \ \Bup\ \ \null} \Bup(xz,yw)+\Bup\bigl(xt(z),t(y)w\bigr)+\Bup(yz,xw)\\
 &\hspace*{150pt}+\Bup\bigl(yt(z),t(x)w\bigr)
    +\Bup\bigl(t(y)z,t(x)w\bigr).
\end{split}
\]
But we have
\begin{multline*}
\Bup\bigl(yt(z),t(x)w\bigr)+\Bup\bigl(t(y)z,t(x)w\bigr)=\Bup\bigl(t(yz),t(x)w\bigr)\\
=\Bup\bigl(yz,t(t(x)w)\bigr)
 =\Bup\bigl(yz,t(x)t(w)\bigr),
\end{multline*}
because $\Bup$, being a morphism in $\Ver_4^+$, satisfies \eqref{eq:Btinvariant}.

Therefore, our condition here is equivalent to the condition
\begin{equation}\label{eq:BBxyzw}
\begin{split}
\Bup(x,y)\Bup(z,w)
 &=\Bup(xz,yw)+\Bup(xw,yz)\\
 &\qquad\qquad +\Bup\bigl(xt(z),t(y)w\bigr)+\Bup\bigl(yz,t(x)t(w)\bigr),
\end{split}
\end{equation}
for all $x,y,z,w\in U$.
\end{enumerate}

\begin{remark}\label{re:kert_perp_imt}
Equation \eqref{eq:Btinvariant} implies, in particular,
that $\im t$ and $\ker t$ are orthogonal relative to $\Bup$:
\[
\Bup(\ker t,\im t)=0.
\]
In particular, $\im t$ is a totally isotropic subspace of $\Bup$.
\end{remark}

\begin{remark}\label{re:symmetryxy_zw}
Equation \eqref{eq:BBxyzw} is symmetric in $x$ and $y$, because so is its left hand side, and also in $z$ and $w$. 

The symmetry in $x$ and $y$ of the right hand side can also be deduced as follows:
\[
\begin{split}
\Bup\bigl(xt(z),&t(y)w\bigr)+\Bup\bigl(yz,t(x)t(w)\bigr)\\
 &=\Bup\bigl(t(xz),t(y)w\bigr)+\Bup\bigl(t(x)z,t(y)w\bigr)+\Bup\bigl(t(yz),t(x)w\bigr)\\
 &=\Bup\bigl(xz,t(y)t(w)\bigr)+\Bup\bigl(yt(z),t(x)w\bigr)+\Bup\bigl(t(y)z,t(x)w\bigr)+\Bup\bigl(t(x)z,t(y)w\bigr)
\end{split}
\]
and $\Bup\bigl(t(y)z,t(x)w\bigr)+\Bup\bigl(t(x)z,t(y)w\bigr)=\Bup\bigl(t(y),t(x)\bigr)\Bup(z,w)=0$, because of 
(the linearization of) \eqref{eq:QBxuv} and since $\im t$ is totally isotropic.

In the same vein, the symmetry in $z$ and $w$ follows because
\[
\begin{split}
\Bup\bigl(xt(z),&t(y)w\bigr)+\Bup\bigl(yz,t(x)t(w)\bigr)\\
 &=\Bup\bigl(xt(z),t(yw)\bigr)+\Bup\bigl(xt(z),yt(w)\bigr)+\Bup\bigl(t(yz),xt(w)\bigr)\\
 &=\Bup\bigl(t(x)t(z),yw\bigr)+\Bup\bigl(xt(z),yt(w)\bigr)+\Bup\bigl(yt(z),xt(w)\bigr)+\Bup\bigl(t(y)z,xt(w)\bigr)\\
 &=\Bup\bigl(yw,t(x)t(z)\bigr)+\Bup(x,y)\Bup\bigl(t(z),t(w)\bigr)+\Bup\bigl(xt(w),t(y)z\bigr)\\
 &=\Bup\bigl(yw,t(x)t(z)\bigr)+\Bup\bigl(xt(w),t(y)z\bigr),
\end{split}
\]
where we have used \eqref{eq:BQuvx} and $\Bup\bigl(t(z),t(w)\bigr)=0$.
\end{remark}

\bigskip

\begin{proposition}\label{pr:adjunction}
Let $(U,\mu,1)$ be a unital algebra in the symmetric tensor category $\Ver_4^+$
over a field $\FF$ of characteristic $2$. Let $\Qup$ be a multiplicative quadratic form (relative to $\mu$)
and let $\Bup$ be the associated symmetric bilinear form. For $x\in U$, write
$\overline{x}=x+\Bup(x,1)1$. Then, for all $x,y,z\in U$, the following equations hold:
\begin{equation}\label{eq:adjoint}
\begin{split}
\Bup(xy,z)&=\Bup\bigl(y,\overline{x}z+t(x)t(z)\bigr)\\ 
\Bup(xy,z)&=\Bup\bigl(x,z\overline{y}+t(z)t(y)\bigr).
\end{split}
\end{equation}
\end{proposition}
\begin{proof}
Put $w=1$ in \eqref{eq:BBxyzw}, and use $t(1)=0$ to get
\[
\Bup(x,y)\Bup(z,1)=\Bup(xz,y)+\Bup(x,yz)+\Bup\bigl(xt(z),t(y)\bigr).
\]
Note that $\Bup\bigl(xt(z),t(y)\bigr)=\Bup\bigl(t(x)t(z),y\bigr)$ because of \eqref{eq:Btinvariant}, and hence
the last equation can be written as
\[
\Bup(x,yz)=\Bup(x\overline{z}+t(x)t(z),y\bigr),
\]
for all $x,y,z\in U$, which is equivalent to the second equation to be proved. 

In the same vein, put $y=1$ in \eqref{eq:BBxyzw} to get
\[
\Bup(x,1)\Bup(z,w)=\Bup(xz,w)+\Bup(xw,z)+\Bup\bigl(z,t(x)t(w)\bigr),
\]
which can be written as
\[
\Bup(xz,w)=\Bup\bigl(z,\overline{x}w+t(x)t(w)\bigr),
\]
for all $x,z,w\in U$, and this is equivalent to the first equation to be proved.
\end{proof}

\smallskip

This result allows us to study the subalgebra $\ker t$ of any unital composition algebra in $\Ver_4^+$.

\begin{theorem}\label{th:kertconic}
Let $(U,\mu,\Qup,1)$ be a unital composition algebra in the symmetric tensor category $\Ver_4^+$
over a field $\FF$ of characteristic $2$. Let $\Bup$ be the associated symmetric bilinear form.
Then the subalgebra $\ker t$ is a conic alternative algebra relative to the restriction of $\Qup$ to $\ker t$.
Moreover, $\im t$ is a commutative ideal of $\ker t$ and concides with the radical of the restriction of the
bilinear form $\Bup$ to $\ker t$.
\end{theorem}
\begin{proof}
As $\Bup$ is nondegenerate and $\Bup(\ker t,\im t)=0$ (Remark \ref{re:kert_perp_imt}), it follows that $\im t$
is the radical of the restriction of $\Bup$ to $\ker t$. Note also that \eqref{eq:QBxuv} gives $\Qup(1)=1$. Also,
as $t$ acts as a derivation, $\im t$ is an ideal of $\ker t$ (Proposition \ref{pr:t1kertimt}).

For any $x\in\ker t$ and $y,z\in U$, Equations \eqref{eq:adjoint} and \eqref{eq:QBxuv} give
\[
\Bup\bigl(\overline{x}(xy),z\bigr)=\Bup(xy,xz)=\Qup(x)\Bup(y,z)=\Bup\bigl(\Qup(x)y,z\bigr).
\]
A similar argument gives $\Bup\bigl((yx)\overline{x},z\bigr)=\Bup\bigl(\Qup(x)y,z\bigr)$.
The nondegeneracy of $\Bup$ then gives
\begin{equation}\label{eq:barxxy}
\overline{x}(xy)=\Qup(x)y=(yx)\overline{x}
\end{equation}
for all $x\in\ker t$ and $y\in U$.

With $y=1$ this gives $\overline{x}x=\Qup(x)1$, that is,
\begin{equation}\label{eq:conic}
x^2+\Bup(x,1)x+\Qup(x)1=0
\end{equation}
for all $x\in\ker t$, proving that $\ker t$ is a conic algebra relative to restriction of $\Qup$, whose polar
form is the restriction of $\Bup$ to $\ker t$. 

Besides,
\eqref{eq:barxxy} and \eqref{eq:conic} imply 
\begin{equation}\label{eq:xxy}
x(xy)=x^2y,\qquad (yx)x=yx^2
\end{equation}
for all $x\in \ker t$ and $y\in U$, and this shows that $\ker t$ is an alternative algebra.

Finally, \eqref{eq:conic} gives $x^2=\Qup(x)1$ for any $x\in \im t$, and hence $xy+yx=\Bup(x,y)1=0$ for 
all $x,y\in \im t$, as $\im t$ is a totally isotropic subspace relative to $\Bup$. Hence $xy=yx$ for all $x,y\in \im t$.
\end{proof}

\bigskip

We will need some further consequences of \eqref{eq:BBxyzw}.

\begin{proposition}\label{pr:adjunction2}
Let $(U,\mu,\Qup,1)$ be a unital composition algebra in the symmetric tensor category $\Ver_4^+$
over a field $\FF$ of characteristic $2$. Let $\Bup$ be the associated symmetric bilinear form. For $x\in U$, write
$\overline{x}=x+\Bup(x,1)1$. Then, for all $x,y,z\in U$, the following equations hold:
\begin{equation}\label{eq:adjoint2}
\begin{split}
(xy)\overline{z}+(xz)\overline{y}&=\Bup(y,z)x+\bigl(xt(y)\bigr)t(z),\\
\overline{x}(yz)+\overline{y}(xz)&=\Bup(x,y)z+t(x)\bigl(t(y)z\bigr).
\end{split}
\end{equation}
\end{proposition}

Note that the expression $\bigl(xt(y)\bigr)t(z)$ is symmetric in $y,z$, because $\overline{u}=u$ for any $u\in\im t$,
and hence the linearization of \eqref{eq:barxxy} gives, using that $\im t$ is totally isotropic, 
$\bigl(xt(y)\bigr)t(z)+\bigl(xt(z)\bigr)t(y)=0$. In the same vein, the expression $t(x)\bigl(t(y)z\bigr)$ is symmetric
in $x,y$.

\begin{proof}
Consider \eqref{eq:BBxyzw} and note that, because of \eqref{eq:adjoint}, one has:
\[
\begin{split}
\Bup(x,y)\Bup(z,w)&=\Bup\bigl(x,\Bup(z,w)y\bigr),\\
\Bup(xz,yw)&=\Bup\bigl(x,(yw)\overline{z}+t(yw)t(z)\bigr),\\
\Bup(xw,yz)&=\Bup\bigl(x,(yz)\overline{w}+t(yz)t(w)\bigr),\\
\Bup\bigl(xt(z),t(y)w\bigr)&=\Bup\bigl(x,(t(y)w)t(z)\bigr),\\
\Bup\bigl(yz,t(x)t(w)\bigr)&=\Bup\bigl(xt(w),t(yz)\bigr)=\Bup\bigl(x,t(yz)t(w)\bigr),
\end{split}
\]
where in the last equation we have used that $\overline{x}=x$ for any $x\in \im t$, because $\im t$ is orthogonal
to $1$ by Remark \ref{re:kert_perp_imt}.

The nondegeneracy of $\Bup$ gives, using that $t(yw)+t(y)w=yt(w)$,
\[
\Bup(z,w)y=(yw)\overline{z}+(yz)\overline{w}+\bigl(yt(w)\bigr)t(z),
\]
which is equivalent to the first equation in \eqref{eq:adjoint2}.

Similarly,
\[
\begin{split}
\Bup(x,y)\Bup(z,w)&=\Bup\bigl(\Bup(x,y)z,w\bigr),\\
\Bup(xz,yw)&=\Bup\bigl(\overline{y}(xz)+t(y)t(xz),w\bigr),\\
\Bup(xw,yz)&=\Bup\bigl(w,\overline{x}(yz)+t(x)t(yz)\bigr),\\
\Bup\bigl(xt(z),t(y)w\bigr)&=\Bup\bigl(t(y)(xt(z)),w\bigr),\\
\Bup\bigl(yz,t(x)t(w)\bigr)&=\Bup\bigl(t(yz),t(x)w\bigr)=\Bup\bigl(t(x)t(yz),w\bigr),
\end{split}
\]
and the nondegeneracy of $\Bup$ gives the second equation in \eqref{eq:adjoint2}.
\end{proof}

\begin{remark}
We may go in the reverse direction, and hence if the second equation in \eqref{eq:adjoint} and the first equation
in \eqref{eq:adjoint2} (or, alternatively, the first equation in \eqref{eq:adjoint} and the second equation
in \eqref{eq:adjoint2} are satisfied, then so is \eqref{eq:BBxyzw}.
\end{remark}

\bigskip

\subsection{Composition algebras with trivial Frobenius twist}
\ \null

Let $(U,\mu,\Qup,1)$ be a unital composition algebra in the symmetric tensor category $\Ver_4^+$
over a field $\FF$ of characteristic $2$ with trivial Frobenius twist: $U^{(1)}=0$. This is equivalent,
thanks to \cite[Proposition 3.4 and Corollary 3.5]{Kannanetal} to the condition $\ker t=\im t$. 
Write $X=\im t$. Then $1\in \ker t=\im t=X$, 
so there
is an element $v\in U$ such that $t(v)=1$. Let $\Bup$ be the associated bilinear form, and consider
the subspace $X_0=\{x\in X\mid \Bup(v,x)=0\}$.

\begin{lemma}\label{le:X0}
Under the conditions above, the following assertions hold:
\begin{itemize}
\item $X_0$ does not depend on $v$,  $X=\FF 1\oplus X_0$, and $U=X\oplus vX$.
\item $vx+xv=x$ for any $x\in X_0$.
\item For any $u\in U$, the following equation holds
\begin{equation}\label{eq:Buu}
\Bup(u,u)=\Bup\bigl(u,t(u)\bigr)=\Qup\bigl(t(u)\bigr).
\end{equation}
\item The subset $\{u\in U\mid t(u)=1\ \text{and}\ u^2\in\FF 1\}$ generates a two-dimensional subalgebra 
$S$ of $U$,
which is a twisted form of the algebra of dual numbers.
\end{itemize}
\end{lemma}

\begin{proof}
Note first that $t^{-1}(1)=v+X$, and since $X$ is totally isotropic relative to $\Bup$ 
(Remark \ref{re:kert_perp_imt}),
it turns out that $X_0=\{x\in X\mid \Bup(x,u)=0\ \forall u\in t^{-1}(1)\}$ does not depend on $v$.

Now, $t(v^2)=t(v)v+vt(v)=v+v=0$, so $v^2\in X$ and \eqref{eq:BBxyzw} gives
\begin{multline*}
\Bup(v,1)^2=\Bup(v^2,1)+\Bup(v,v)=\Bup(v,v)\\
=\Qup\bigl(\rho_U(v\otimes v)\bigr)
=\Qup\bigl(t(v)\otimes t(v)\bigr)=\Qup(1)=1.
\end{multline*}
Hence $\Bup(v,1)=\Bup(v,v)=1$, so $1\not\in X_0$ and then $X_0$ is a subspace of $X$ of codimension $1$,
and $X=\FF 1\oplus X_0$. Moreover, $t(vx)=t(v)x+vt(x)=1x+v0=x$, so $t$ induces a bijection $vX\rightarrow X$
and $vX\cap X=0$. Therefore, $U=X\oplus vX$ because $\ker t=\im t$, so that $\dim U/X=\dim X$.

For any $u\in U$,
\[
\Bup\bigl(vt(u),t(u)\bigr)=
\begin{cases} 
\Qup\bigl(t(u)\bigr)\Bup(v,1)=\Qup\bigl(t(u)\bigr),\ \text{(by \eqref{eq:BQuvx})}\\
\Bup\bigl(t(vt(u)),u\bigr)=\Bup\bigl(t(v)t(u),u\bigr)=\Bup\bigl(u,t(u)\bigr),
\end{cases}
\]
and hence we get 
\[
\Bup\bigl(u,t(u)\bigr)=\Qup\bigl(t(u)\bigr)=\Qup\bigl(\rho_U(u\otimes u)\bigr)=\Bup(u,u),
\]
as required.

For any $x\in X_0$, $\overline{x}=x$, and \eqref{eq:adjoint2} gives $\overline{v}x+xv=\Bup(v,x)1=0$.
Since $\overline{v}=v+\Bup(v,1)1=v+1$, we get $vx+xv=x$. 

Since $v^2\in X$, we may write $v^2=\alpha 1+x$ for some $\alpha\in\FF$ and $x\in X_0$. Then 
$t(v+x)=1$ and, because of \eqref{eq:conic}
\[
(v+x)^2=v^2+(vx+xv)+x^2=(\alpha 1+x) +x+\Qup(x)1=\bigl(\alpha+\Qup(x)\bigr)1.
\]
Hence we may replace our original $v$ by $v+x$ and assume that $v^2=\lambda 1$ for some $\lambda\in\FF$.
Using that $vy+yv=y$ for any $y\in X_0$, it follows that the subset 
$\{u\in U\mid t(u)=1\ \text{and}\ u^2\in\FF 1\}$ is precisely $v+\FF 1$, and hence the subalgebra $S$ 
generated by this subset is $\FF 1+\FF v$, which is a twisted form of the algebra of dual numbers.
\end{proof}

\begin{remark}\label{re:lambda_invariant}
With the notation above, $v^2=\lambda 1$, and for any 
$v+\alpha 1\in \{u\in U\mid t(u)=1\ \text{and}\ u^2\in\FF 1\}$, $(v+\alpha 1)^2=(\lambda+\alpha^2)1$,
so that the class of $\lambda$ modulo $\FF^2$ is an invariant of the composition algebra $(U,\mu,\Qup)$.
\end{remark}

Pick then an element $v\in U$ with $t(v)=1$ and $v^2=\lambda 1$ for some $\lambda\in\FF$, and
define the linear map
\begin{equation}\label{eq:beta}
\begin{split}
\beta\colon X&\longrightarrow \FF\\
                 x&\mapsto \Bup(v,x).
\end{split}
\end{equation}
Note that $\beta$ is independent of $v$, because $\im t$ is totally isotropic, and $X_0$ is just the kernel of $\beta$.

\begin{lemma}\label{le:Bbeta}
The bilinear form $\Bup$ is determined by $\beta$. More explicitly, for all $x,y\in X$ we have
\[
\Bup(x,y)=0,\quad \Bup(x,vy)=\beta(xy),\quad \Bup(vx,vy)=\beta(xy).
\]
\end{lemma}

\begin{proof}
We already know that $X$ is totally isotropic (Remark \ref{re:kert_perp_imt}). For $x,y\in X_0$, 
Proposition \ref{pr:adjunction}
gives, as $\overline{y}=y$, $\Bup(x,vy)=\Bup(xy,v)=\beta(xy)$. Now, using \eqref{eq:BBxyzw}, we get
\[
\beta(x)\beta(y)=\Bup(x,v)\Bup(v,y)\\
  =\Bup(xv,vy)+\Bup(xy,v^2)+\Bup(x,y)=\Bup(xv,vy),
\]
because $X$ is totally isotropic. Since $vx+xv=x$ for any $x\in X_0$ and $X=\FF 1\oplus X_0$, it follows
that $vx+xv=x+\beta(x)1$ for any $x\in X$. But $\overline{v}=v+1$, so this can be written as
\begin{equation}\label{eq:vbarra}
vx+x\overline{v}=\beta(x)1=\overline{v}x+xv
\end{equation} 
for any $x\in X$. Then, an easy computation gives
\[
\begin{split}\Bup(vx,vy)&=\Bup(x\overline v+\beta(x)1,vy)=\Bup(x+xv+\beta(x)1,vy)\\
  &=\Bup(x,vy)+\Bup(xv,vy)+\beta(x)\Bup(1,vy)\\
  &=\beta(xy)+\beta(x)\beta(y)+\beta(x)\beta(y)=\beta(xy).\qedhere
\end{split}
\]
\end{proof}

\begin{lemma}\label{le:beta_associative}
The bilinear form $X\times X\rightarrow \FF$, $(x,y)\mapsto \beta(xy)$, is symmetric, nondegenerate and
associative (i.e., $\beta\bigl((xy)z\bigr)=\beta\bigl(x(yz)\bigr)$ for all $x,y,z\in X$). Moreover, $X_0$ is
the orthogonal complement to $\FF 1$ for this bilinear form.
\end{lemma}

\begin{proof}
The multiplication in $X$ is commutative (Theorem \ref{th:kertconic}), so $\beta(xy)$ is a symmetric bilinear form.
Since $\Bup$ is nondegenerate, $X$ is totally isotropic, and $U=X\oplus vX$, the restriction of $\Bup$ to 
$X\times vX$ is nondegenerate too, and Lemma \ref{le:Bbeta} shows then that $\beta(xy)$ is nondegenerate.

Finally, \eqref{eq:adjoint} gives
$\beta\bigl((xy)z\bigr)=\Bup\bigl((xy)z,v\bigr)=\Bup(xy,vz)=\Bup\bigl(x,(vz)y\bigr)$, for $x,y,z\in X$.
But $vz+z\overline{v}=\beta(z)1$ by \eqref{eq:vbarra}, hence we have
$(vz)y=(zv)y+zy+\beta(z)y$, and $(zv)y+(zy)\overline{v}=\Bup(v,y)z=\beta(y)z$ by \eqref{eq:adjoint2},
so that $(zv)y=(zy)v+zy+\beta(y)z$ and 
\[
(vz)y=(zv)y+zy+\beta(z)y=(zy)v+\beta(y)z+\beta(z)y.
\]
Using the commutativity of the multiplication in $X$ and Lemma \ref{le:Bbeta} we get
\[
\beta\bigl((xy)z\bigr)=\Bup\bigl(x,(vz)y\bigr)=\Bup\bigl(x,(zy)v\bigr)=\beta\bigl(x(zy)\bigr)=\beta\bigl(x(yz)\bigr),
\]
as required.
\end{proof}

Because of the associativity of $\beta$, we may write $\beta(xyz)$.

The next result restricts severely the dimension of any unital composition algebra in $\Ver_4^+$ with 
$\ker t=\im t$.

\begin{theorem}\label{th:dimX2}
Let $(U,\mu,\Qup,1)$ be a unital composition algebra in the symmetric tensor category $\Ver_4^+$
over a field $\FF$ of characteristic $2$ with trivial Frobenius twist ($\ker t=\im t$). Then the dimension of 
$\im t$ is at most two.
In other words, the dimension of $U$ is either $2$ or $4$.
\end{theorem}

\begin{proof}
With the notations above, where $v\in U$ satisfies $t(v)=1$ and $v^2=\lambda 1$, $\lambda\in \FF$,
note that for any $x\in X_0$, \eqref{eq:adjoint2} gives $\overline{v}(xv)+xv^2=0$, so that, as 
$\overline{v}=v+1$, $v(xv)=xv+\lambda x$ and hence, using Lemma \ref{le:X0}
\begin{equation}\label{eq:vvx}
v(vx)=v(xv+x)=xv+\lambda x+vx=(\lambda +1)x.
\end{equation}
Now, note that $\overline{vx}=vx$ because $\Bup(vx,1)=\beta(x)=0$ (Lemma \ref{le:Bbeta}), and hence
\eqref{eq:adjoint2} gives, for $x,y\in X_0$,
$(vx)(vy)+(v(vy))x=\Bup(x,vy)v=\beta(xy)v$, so that
\begin{equation}\label{eq:vxvy}
(vx)(vy)=\beta(xy)v+(\lambda +1)xy.
\end{equation}

Equation \eqref{eq:adjoint} gives, for $x,y,z\in X_0$,
\[
\Bup\bigl((vx)(vy),vz\bigr)=\Bup\bigl(vx,(vz)(vy)+zy\bigr).
\]
The left hand side of this equation is, thanks to Lemma \ref{le:Bbeta},
\[
\begin{split}
\Bup\bigl((vx)(vy),vz\bigr)&=\Bup\bigl(\beta(xy)v+(\lambda +1)xy,vz\bigr)\\
 &=\beta(xy)\Bup(v,vz)+(\lambda +1)\Bup(xy,vz)\\
 &=\beta(xy)\beta(z)+(\lambda +1)\beta(xyz)=(\lambda+1)\beta(xyz),\quad \text{as $\beta(z)=0$,}
\end{split}
\]
while the right hand side is
\[
\begin{split}
\Bup\bigl(vx,(vz)(vy)+zy\bigr)&=\Bup\bigl(vx,\beta(zy)v+(\lambda+1)zy+zy\bigr)\\
  &=\beta(zy)\Bup(vx,v)+\lambda\Bup(vx,zy)=\lambda\beta(xzy)=\lambda\beta(xyz).
\end{split}
\]
We conclude that $\beta(xyz)=0$ for any $x,y,z\in X_0$ and, by the nondegeneracy of $\beta$ 
(Lemma \ref{le:beta_associative}), this implies that $xy\in \FF1$ for any $x,y\in X_0$, so we have
\begin{equation}\label{eq:xy}
xy=\beta(xy)1
\end{equation}
for all $x,y\in X_0$. Multiplying this equation on the left by $x$ and using \eqref{eq:xxy} and \eqref{eq:conic}
we get $\Qup(x)y=x^2y=x(xy)=\beta(xy)x$. Hence, if $X_0\neq 0$, the nondegeneracy of $\beta$ (Lemma
\ref{le:beta_associative}) gives that if $\Qup(x)=0$, then necessarily $x=0$, and hence for any $0\neq x\in X_0$,
we have that $X_0=\FF x$. Thus either $X_0=0$ or $X_0$ is one-dimensional, as required.
\end{proof}

\bigskip

We are ready to classify the unital composition algebras on $\Ver_4^+$ with $\ker t=\im t$, that is, with
trivial Frobenius twist. Recall that, in this case, we have $\Gamma^2(U)=\rho_U(U\otimes U)$, 
so the quadratic form $\Qup$
is determined by $\Bup$, as $\Bup(u,v)=\Qup\bigl(\rho_U(u\otimes v)\bigr)$.

\begin{theorem}\label{th:kertimtclass}
Let $(U,\mu,\Qup,1)$ be a unital composition algebra in the symmetric tensor category $\Ver_4^+$
over a field $\FF$ of characteristic $2$ with trivial Frobenius twist ($\ker t=\im t$).
\begin{itemize}
\item 
If the dimension of $U$ is $2$, then there is an element $v\in U$ such that $t(v)=1$ and $v^2=\lambda 1$
for some $\lambda\in\FF$. The bilinear form $\Bup$ satisfies 
\[
\Bup(v,v)=\Bup(v,1)=1,\quad \Bup(1,1)=0.
\]

Denote by $U(\lambda)$ this algebra. 
The algebras $U(\lambda)$ and $U(\lambda')$
are isomorphic if and only if $\lambda'-\lambda\in\FF^2$.

\item If the dimension of $U$ is $4$, there is an element $v\in U$ such that $t(v)=1$ and $v^2=\lambda 1$
for some $\lambda\in\FF$ and an element $x\in \im t$ with $\Bup(v,x)=0$ such that $(1,v,x,vx)$ is a basis of
$U$ and the multiplication is given by the following table, where $0\neq \nu\in\FF^\times$:
\[
\vcenter{\offinterlineskip
\halign{\hfil$#$\enspace\hfil&#\vreglon
 &\hfil\enspace$#$\enspace\hfil
 &\hfil\enspace$#$\enspace\hfil
 &\hfil\enspace$#$\enspace\hfil
 &\hfil\enspace$#$\enspace\hfil\cr
 &\omit\hfil\vrule width 1pt depth 4pt height 10pt
   &1&v&x&vx \cr
 \noalign{\hreglon}
 1&&1&v& x&xv \cr
 v&&v&\lambda 1&vx&(\lambda+1)x \cr
 x&&x&x+vx&\nu 1&\nu(v+1) \cr
 vx&&vx&\lambda x+vx&\nu v&\nu v+(\lambda +1)\nu 1 \cr
}}
\]
The bilinear form is given by
\[
\Bup(v,v)=\Bup(v,1)=1,\quad \Bup(vx,vx)=\Bup(vx,x)=\nu,
\]
and all the other values are either $0$ or obtained by the symmetry of $\Bup$.

Denote by $U(\lambda,\nu)$ this algebra. The algebras $U(\lambda,\nu)$ and $U(\lambda',\nu')$ are 
isomorphic (as algebras in $\Ver_4^+$!) if and only if $\lambda'-\lambda\in\FF^2$ and $\nu'\nu^{-1}\in (\FF^\times)^2$.
\end{itemize}
\end{theorem}

\begin{proof}
If the dimension is $2$, the algebra is spanned by $1$ and $v$ with $t(v)=1$, and as in Lemma \ref{le:X0},
$v^2=\lambda 1$ for some $\lambda\in\FF$. Besides, \eqref{eq:Buu} gives $\Bup(v,v)=\Bup(v,1)=1$, while
$\Bup(1,1)=0$ as $X=\FF 1$ is totally isotropic.

It is straightforward to check that this is indeed a composition
algebra in $\Ver_4^+$, but this will also follow from the next case, where this subalgebra appears as a subalgebra.

\smallskip

If the dimension is $4$, then the dimension of $X_0$ is $1$ and, as in the proof of Theorem \ref{th:dimX2}, 
$X_0=\FF x$ for an element $x$ with $\Qup(x)=\nu\neq 0$. The subalgebra $S$ in Lemma
\ref{le:X0} is generated by an element $v$ with $t(v)=1$ and $v^2=\lambda 1$ for some $\lambda\in\FF$.
Then $U=X\oplus vX$ (Lemma \ref{le:X0})
is spanned by the elements $1,v,x,vx$. Note that $x^2=\nu 1$ by \eqref{eq:conic}, $v(vx)=(\lambda +1)x$
by \eqref{eq:vvx}, $xv=vx+x$ by Lemma \ref{le:X0}, $x(vx)=x(xv+x)=\nu v+\nu 1$, because of 
\eqref{eq:xxy}, which also gives $(vx)x=\nu v$, $(vx)^2=\nu v+ (\lambda +1)\nu 1$ because of
\eqref{eq:vxvy}, and $v^2x+(vx)\overline{v}=0$ by \eqref{eq:adjoint2}, so $(vx)v=vx+\lambda x$, thus
completing the multiplication table above. 

The bilinear form $\Bup$ is determined by Lemma \ref{le:Bbeta}:  $\Bup(X,X)=0$, $\Bup(v,v)=\Bup(v,1)=1$, 
$\Bup(vx,vx)=\Bup(vx,x)=\beta(x^2)=\nu$,  and $\Bup(vx,1)=\Bup(v,x)=\beta(x)=0=\Bup(vx,v)$.

The isomorphism condition follows from the invariance of the subalgebra $S$ in Lemma \ref{le:X0} and of
the subspace $X_0$. The fact that $t$ acts as a derivation follows at once.

It remains to show that indeed the algebra with multiplication table above is a composition
algebra in $\Ver_4^+$. For this we may extend scalars to the algebraic closure, and hence, substituting
$v$ by $v+\sqrt{\lambda}1$ and $x$ by $\frac{1}{\sqrt{\nu}}x$, we may assume that 
$\lambda=0$ and $\nu=1$. The composition property $\Qup\otimes\Qup=\Qup\circ\tilde\mu$ is equivalent
to the fact that $\Qup\colon\Gamma^2(U)\rightarrow \FF$ is a homomorphism of algebras. The object
$\Gamma^2(U)$ is spanned by the following elements (indicating their image under 
$\Qup\colon\Gamma^2(U)\rightarrow \FF$):
\[
\begin{aligned}
I&=1\otimes 1\xrightarrow{\Qup} 1,&W&=vx\otimes 1+1\otimes vx\xrightarrow{\Qup}0,\\
X&=x\otimes x\xrightarrow{\Qup} 1,&V_x&=v\otimes x+x\otimes v\xrightarrow{\Qup} 0,\\
I_x&=x\otimes 1+1\otimes x\xrightarrow{\Qup} 0,\qquad&W_x&=vx\otimes x+x\otimes vx\xrightarrow{\Qup} 1,\\
V&=v\otimes 1+1\otimes v\xrightarrow{\Qup} 1,&Z&=v\otimes vx+vx\otimes v +x\otimes 1\xrightarrow{\Qup} 0.
\end{aligned}
\]
For instance, $\Qup(V)=\Qup\circ\rho_U(v\otimes 1)=\Bup(v,1)=1$.

\smallskip

The kernel of $\Qup$ is spanned by the elements $X-I,I_x,V-I,W,V_x,W_x-I,Z$. 

\smallskip

The multiplication table
is easily computed:
{\tiny\[
\vcenter{\offinterlineskip
\halign{\hfil$#$\enspace\hfil&#\vreglon
 &\hfil\enspace$#$\enspace\hfil
 &\hfil\enspace$#$\enspace\hfil
 &\hfil\enspace$#$\enspace\hfil
 &\hfil\enspace$#$\enspace\hfil
 &\hfil\enspace$#$\enspace\hfil
 &\hfil\enspace$#$\enspace\hfil
 &\hfil\enspace$#$\enspace\hfil
 &\hfil\enspace$#$\enspace\hfil\cr
 &\omit\hfil\vrule width 1pt depth 4pt height 10pt
   &I&X&I_x&V&W&V_x&W_x&Z \cr
 \noalign{\hreglon}
 I&&I&X&I_x&V&W&V_x&W_x&Z \cr
 X&&X&I&I_x&W_x&I_x+V_x&I_x+W&V&V_x+W+Z\cr
 I_x&&I_x&I_x&0&I_x+W+V_x&V+W_x&V+W_x&I_x+W+V_x&I+X+V+W_x\cr
 V&&V&W_x&W+V_x&I&I_x+Z&I_x+Z&X&W+V_x\cr
 W&&W&V_x&V+W_x&I_x+W+Z&X+V&I+W_x&I_x+V_x+Z&I+X\cr
 V_x&&V_x&W&V+W_x&V_x+Z&I+V&X+W_x&W+Z&0\cr
 W_x&&W_x&V&W+V_x&X&W+V_x+Z&W+V_x+Z&I&0\cr
 Z&&Z&I_x+Z&I+X&V_x+Z&I+V&X+V&I_x+V_x+Z&I+V\cr
}}
\]}
(For instance, $Z^2=\tilde\mu\bigl((v\otimes vx+vx\otimes v+x\otimes 1)
\otimes(v\otimes vx+vx\otimes v+x\otimes 1)\bigr)$,
and 
\[
\begin{split}
\tilde\mu\bigl((v\otimes vx)\otimes(v\otimes vx)\bigr)
   &=(\mu\otimes\mu)\circ c_{23}(v\otimes vx\otimes v\otimes vx)\\
   &=(\mu\otimes\mu)(v\otimes v\otimes vx\otimes vx+v\otimes 1\otimes x\otimes vx)\\
   &=v^2\otimes (vx)^2+v\otimes x(vx)=v\otimes(v+1),\\
\tilde\mu\bigl((v\otimes vx)\otimes(vx\otimes v)\bigr)
   &=(\mu\otimes\mu)\circ c_{23}(v\otimes vx\otimes vx\otimes v)\\
   &=(\mu\otimes\mu)(v\otimes vx\otimes vx\otimes v+v\otimes x\otimes x\otimes v)\\
   &=v(vx)\otimes (vx)v+vx\otimes xv=x\otimes vx+ vx\otimes (x+vx),\\
\tilde\mu\bigl((v\otimes vx)\otimes(x\otimes 1)\bigr)
   &=(\mu\otimes\mu)\circ c_{23}(v\otimes vx\otimes x\otimes 1)\\
   &=(\mu\otimes\mu)(v\otimes x\otimes vx\otimes 1)=vx\otimes vx,
\end{split}
\]
\[
\begin{split}
\tilde\mu\bigl((vx\otimes v)\otimes(v\otimes vx)\bigr)
   &=(\mu\otimes\mu)\circ c_{23}(vx\otimes v\otimes v\otimes vx)\\
   &=(\mu\otimes\mu)(vx\otimes v\otimes v\otimes vx+vx\otimes 1\otimes 1\otimes vx)\\
   &=(vx)v\otimes v(vx)+vx\otimes vx=vx\otimes x+ vx\otimes vx,\\
\tilde\mu\bigl((vx\otimes v)\otimes(vx\otimes v)\bigr)
   &=(\mu\otimes\mu)\circ c_{23}(vx\otimes v\otimes vx\otimes v)\\
   &=(\mu\otimes\mu)(vx\otimes vx\otimes v\otimes v+vx\otimes x\otimes 1\otimes v)\\
   &=(vx)^2\otimes v^2+(vx)x\otimes v=v\otimes v,\\
\tilde\mu\bigl((vx\otimes v)\otimes(x\otimes 1)\bigr)
   &=(\mu\otimes\mu)\circ c_{23}(vx\otimes v\otimes x\otimes 1)\\
   &=(\mu\otimes\mu)(vx\otimes x\otimes v\otimes 1=v\otimes v,
\end{split}
\]
\[
\begin{split}
\tilde\mu\bigl((x\otimes 1)\otimes(v\otimes vx)\bigr)
   &=(\mu\otimes\mu)\circ c_{23}(x\otimes 1\otimes v\otimes vx)\\
   &=(\mu\otimes\mu)(x\otimes v\otimes 1\otimes vx)=xv\otimes vx=(x+vx)\otimes vx,\\
\tilde\mu\bigl((x\otimes 1)\otimes(vx\otimes v)\bigr)
   &=(\mu\otimes\mu)\circ c_{23}(x\otimes 1\otimes vx\otimes v)\\
   &=(\mu\otimes\mu)(x\otimes vx\otimes 1\otimes v)=x(vx)\otimes v=(v+1)\otimes v,\\
\tilde\mu\bigl((x\otimes 1)\otimes(x\otimes 1)\bigr)
   &=(\mu\otimes\mu)\circ c_{23}(x\otimes 1\otimes x\otimes 1)\\
   &=(\mu\otimes\mu)(x\otimes x\otimes 1\otimes 1)=x^2\otimes 1=1\otimes 1,
\end{split}
\]
so we get
\[
\begin{split}
Z^2&=v\otimes(v+1)+x\otimes vx+ vx\otimes (x+vx)+vx\otimes vx+ 
vx\otimes x+ vx\otimes vx\\
 &\qquad\qquad +v\otimes v + v\otimes v+(x+vx)\otimes vx +(v+1)\otimes v+ 1\otimes 1\\
&=v\otimes 1+1\otimes v+1\otimes 1=I+V.
\end{split}
\]

By inspection, in the multiplication table above, the rows starting with an element $A$ with $\Qup(A)=0$ have all 
the elements $B$ with $\Qup(B)=0$, while the rows starting with an element $A$ with $\Qup(A)=1$ satisfy that
$\Qup(AB)=\Qup(B)$ for all $B$ in our basis. Hence $\Qup\colon\Gamma^2(U)\rightarrow \FF$ is
an algebra homomorphism, and hence $U$ is indeed a composition algebra in $\Ver_4^+$.
\end{proof}

\begin{remark}
For the algebra $U=U(\lambda)$ we have $\Gamma^2(U)=\FF (1\otimes 1)+\FF(v\otimes 1+1\otimes v)$, and
the multiplicative quadratic form is given by $\Qup(1\otimes 1)=1$, $\Qup(v\otimes 1+1\otimes 1)=1$, which is
independent on $\lambda$! Therefore, nonisomorphic unital composition algebras may have isometric quadratic
forms, something that does not happen for the unital composition algebras in $\vect$.

Something similar can be said of the $U(\lambda,\nu)$'s.
\end{remark}

Over an algebraically closed field there are only two possibilities.

\begin{corollary}\label{co:kertimtFbar}
Let $(U,\mu,\Qup,1)$ be a unital composition algebra in the symmetric tensor category $\Ver_4^+$
over an algebraically field $\FF$ of characteristic $2$ with trivial Frobenius twist. Then either $U$ is isomorphic
to $U(0)$ or to $U(0,1)$.
\end{corollary}

\bigskip

\subsection{Composition algebras with nontrivial Frobenius twist}
\ \null

Our goal in this subsection is to prove the following result.

\begin{theorem}\label{th:kert_not_imt}
Let $(U,\mu,\Qup,1)$ be a unital composition algebra in the symmetric tensor category $\Ver_4^+$
over a field $\FF$ of characteristic $2$ with nontrivial Frobenius twist ($\ker t\neq\im t$). Then $t=0$ and, 
therefore, $(U,\mu,\Qup,1)$
is a usual composition algebra, i.e., a composition algebra in $\vect$.
\end{theorem}

\begin{proof}
We may extend scalars and assume that the ground field is algebraically closed. Assume that $t$ is not trivial. 
We will follow several steps to get a contradiction:

\smallskip

\noindent\textbf{1.} \emph{The unity lies in $\ker t\setminus \im t$ and $\im t$ is a commutative, associative and 
nilpotent subalgebra. Moreover, $x^2=0$ for any $x\in\im t$.} 

If $1$ lied in $\im t$, since this is an ideal of 
$\ker t$ (Theorem \ref{th:kertconic}), we would get $\ker t=\im t$. Therefore, $1\in\ker t\setminus \im t$.
We already know that $\im t$ is commutative by Theorem \ref{th:kertconic}. Now, by \eqref{eq:adjoint2} and commutativity,
\[
(xy)z=(xz)y=(zx)y=(zy)x=x(yz),
\]
for all $x,y,z\in \im t$. Moreover, \eqref{eq:conic} gives $x^2=\Qup(x)1\in \im t\cap\FF 1=0$ for 
any $x\in \im t$,  so the restricition of $\Qup$ to $\im t$ is trivial and $x^2=0$ for all 
$x\in \im t$. (Note that the anticommutativity property: $x^2=0$ for all $x\in \im t$, is stronger than
commutativity since the characteristic of $\FF$ is $2$.) Now, any finite-dimensional associative algebra satisfying $x^2=0$ for all $x$ is nilpotent, as any product
of $\dim X+1$ elements of a basis of $X$ has a repeated element.

\medskip

\noindent\textbf{2.} \emph{$U$ contains a subalgebra $T$ isomorphic to the Cartesian product of two copies of 
$\FF$.} 

As $\im t$ is the radical of the restriction of $\Bup$ to $\ker t$ by Theorem \ref{th:kertconic}, the nondegeneracy of 
$\Bup$ implies the existence of an element $u\in \ker t$ with $\Bup(1,u)=1$. Then $u^2=u+\Qup(u)1$ 
by \eqref{eq:conic}. Thus $T\bydef \FF 1\oplus \FF u$ is a subalgebra of $\ker t$ and $T\cap \im t=0$, as $\Bup$ is 
nondegenerate on $T$. Thus $T$ is a $2$-dimensional unital composition algebra in $\vect$ and, since the ground field is assumed to be algebraically closed, it is isomorphic to the Cartesian product of two copies of $\FF$.

\medskip

\noindent\textbf{3.} \emph{$\im t$ has dimension $2$ and $(\im t)^2=0$.} 

To prove this, note that $T$ is spanned by two orthogonal idempotents $e_1$ and $e_2$, 
with $1=e_1+e_2$, that is, $T=\FF e_1\oplus \FF e_2$ with $e_1^2=e_1$, $e_2^2=e_2$ and $e_1e_2=0=e_2e_1$. Moreover,
$\Bup(e_1,e_2)=1$, and $\Qup(e_1)=0=\Qup(e_2)$. Then $U=T\oplus T^\perp$. By \eqref{eq:adjoint}
we get $\Bup(TT^\perp,T)=\Bup(T^\perp, T^2)=0=\Bup(T^\perp T,T)$, because $T$ is a subalgebra,
and this shows that $TT^\perp+T^\perp T$ is contained in $T^\perp$. Note that, as $T$ is contained in
$\ker t$, then $T^\perp$ contains $\im t$.

Moreover, for any $a,b\in T$ and
$z\in T^\perp$, \eqref{eq:adjoint2} gives
$az=z\overline{a}$, and $(ab)z=(az)\overline{b}$, as $\overline{z}=z$. 
But then $(ab)z=(az)\overline{b}=b(az)$, and
since $T$ is commutative, this gives $(ab)z=a(bz)$ for any $a,b\in T$ and $z\in T^\perp$, proving that
$T^\perp$ is an associative left module for $T$, and actually an associative bimodule, as 
$(az)b=\overline{b}(az)=(a\overline{b})z=a(\overline{b}z)=a(zb)$, for all $a,b\in T$ and $z\in T^\perp$.

Note that $\overline{e_1}=e_2$. Thus we get the `Peirce decomposition'
\[
U=\FF e_1\oplus \FF e_2\oplus U_{12}\oplus U_{21},
\]
where
\[
\begin{aligned}
U_{12}&=\{u\in T^\perp\mid e_1u=u\}&U_{21}&=\{u\in T^\perp\mid e_2u=u\}\\
 &=\{u\in U\mid e_1u=u=ue_2\}\quad&&=\{u\in U\mid e_2u=u=ue_1\}\\
 &=\{u\in U\mid e_2u=0=ue_1\},&&=\{u\in U\mid e_1u=0=ue_2\}.
\end{aligned}
\]
Equation \eqref{eq:adjoint2} gives $e_1(xy)=x(e_2y)$ for any $x,y\in T^\perp$, and this implies
$T^\perp U_{12}\subseteq \{u\in U\mid e_1u=0\}=\FF e_2\oplus U_{21}$
and $T^\perp U_{21}\subseteq \{u\in U\mid e_1u=u\}=\FF e_1\oplus U_{12}$.
In the same vein from $(xy)e_1=(xe_2)y$ for $x,y\in T^\perp$ we
obtain $U_{21}T^\perp\subseteq \FF e_2\oplus U_{12}$ and  $U_{12}T^\perp\subseteq \FF e_1\oplus U_{21}$. 
As a consequence, the following relations hold:
\begin{equation}\label{eq:Us}
\begin{aligned}U_{12}^2&\subseteq U_{21}\quad &U_{21}^2&\subseteq U_{12},\\
  U_{12}U_{21}&\subseteq \FF e_1,&U_{21}U_{12}&\subseteq \FF e_2.
\end{aligned}
\end{equation}
Also, for any $x,y\in U_{12}$, $\Bup(x,y)=\Bup(e_1x,e_1y)=\Qup(e_1)\Bup(x,y)=0$, because of \eqref{eq:QBxuv}.
Hence we get $\Bup(U_{12},U_{12})=0$ and, similarly, $\Bup(U_{21},U_{21})=0$, so that
$\Bup$ induces a nondegenerate bilinear map $U_{12}\times U_{21}\rightarrow \FF$, and thus
$\dim U_{12}=\dim U_{21}$. Besides, for any $x\in U_{12}$, $t(x)=t(e_1x)=e_1t(x)$, which forces
$t(U_{12})\subseteq U_{12}$, and also $t(U_{21})\subseteq U_{21}$.

As $\im t$ is nilpotent, the subspace
$X=\{x\in \im t\mid x(\im t)=0\}$ is nonzero. From $(xe_1)y=(xy)e_2$ for any $x,y\in T^\perp$, it follows that
$X=(X\cap U_{12})\oplus (X\cap U_{21})$. Without loss of generality, we may assume $X\cap U_{12}\neq 0$.
Pick a nonzero element $x\in X\cap U_{12}$. There is an element $v\in U_{21}$ with $\Bup(x,v)=1$.
Remark \ref{re:kert_perp_imt} gives $t(v)\neq 0$.
Then for any $z\in \im t$, \eqref{eq:adjoint2} gives $(zx)v+(zv)x=\Bup(x,v)z=z$. But $xz=zx=0$ (recall that
$\im t$ is commutative by Theorem \ref{th:kertconic}), so we get 
\begin{equation}\label{eq:x_xwz}
z=(zv)x
\end{equation}
for any $z\in \im t$. If 
$z\in \im t\, \cap\, U_{12}=t(U_{12})$,
then $zv\in \FF e_1$, and $e_1x=x$, so we conclude  $z=(zv)x\in \FF x$, that is, 
$t(U_{12})=\im t\cap U_{12}=\FF x$. Since $t$ is skew relative to $\Bup$ (equation \eqref{eq:Btinvariant}),
we conclude that $t(U_{21})$ also has dimension $1$, and hence the dimension of $\im t$ is $2$.

Thus $\im t$ is anticommutative, nilpotent and two-dimensional, and this forces $(\im t)^2=0$. 

\medskip

\noindent\textbf{4.} By \eqref{eq:x_xwz} we have $0\neq t(v)=(t(v)v)x$, and the element $a=t(v)v\in U_{21}^2\subseteq U_{12}$
lies in $\ker t$, because $t(a)=t(v)^2=0$, while it is
not a scalar multiple of $x$, as $ax=(t(v)v)x=t(v)\neq 0$, while $x^2=0$. We conclude that the dimension
of $\ker t\cap U_{12}$ is at least $2$. The same happens in $U_{21}$ since $t$ is skew relative to $\Bup$
and $U_{12}$ and $U_{21}$ are paired by $\Bup$.

The computations above give a subspace $\FF u\oplus \FF a\oplus \FF x\subseteq U_{12}$, with $u$ such that $t(u)=x$ and 
chosen, by adding a scalar multiple of $x$ if needed, so that $\Bup(u,v)=0$. Note that 
$u\in U_{12}\setminus \ker t$, $a\in (U_{12}\cap \ker t)\setminus \im t$, and $x\in U_{12}\cap \im t$.

In the same vein, 
with $y=t(v)$ and $b=t(u)u=xu\in\ker t\setminus \im t$, we have $\FF v\oplus \FF b\oplus \FF y\subseteq U_{21}$,
with $v\in U_{21}\setminus \ker t$, $b\in (U_{21}\cap \ker t)\setminus \im t$, and $y\in U_{12}\cap \im t$.

Note that $1=\Bup(x,v)=\Bup(t(u),v)=\Bup(u,t(v))=\Bup(u,y)$, because of \eqref{eq:Btinvariant}.

\medskip

\noindent\textbf{5.} Some easy computations will be carried out now:
\begin{itemize}
\item The element $a=t(v)v=yv$ equals $vy$ and, similarly, $b=xu=ux$.

This follows because $yv+vy=\Bup(y,v)1=0$ by \eqref{eq:adjoint2}, using that $U_{21}$ is isotropic relative to 
$\Bup$.

\smallskip
 
\item $xv=e_1$, $yu=e_2$, $by=x$, and $ax=y$. 

To prove this, note that $xv$ belongs to $\FF e_1$ by \eqref{eq:Us}. Now \eqref{eq:x_xwz} gives $x=(xv)x$,
so necessarily $xv=e_1$. Also, \eqref{eq:x_xwz} gives $y=(yv)x=ax$. The proof of the other two equalities
follows by symmetry.

\smallskip

\item $\Bup(a,b)=1$, as we have $\Bup(a,b)=\Bup(vy,b)=\Bup(v,by)=\Bup(v,x)=1$, where we have used
\eqref{eq:adjoint}.

\smallskip

\item $uv=0=vu$.

Again, $uv\in\FF e_1$ by \eqref{eq:Us} and 
\[
(uv)x=(ux)v+\Bup(x,v)u=(xu)v+u=(xv)u+u=e_1u+u=2u=0,
\]
where we have used \eqref{eq:adjoint2} and the previous items. By symmetry we get $vu=0$.

\smallskip

\item $au=v$ because $au=(yv)u=(yu)v=e_2v=v$ by \eqref{eq:adjoint2}. Symmetrically, $bv=u$.

\smallskip

\item $u^2,v^2\in\ker t$, $\Bup(u^2,v^2)=1$, $\Bup(u^2,a)=0=\Bup(v^2,b)$, $u^2\not\in\FF b$, 
and $v^2\not\in\FF a$:

Indeed, $t(u^2)=t(u)u+ut(u)=2b=0$, so $u^2\in\ker t\cap U_{21}$. Symmetrically $v^2\in\ker t\cap U_{12}$. 
By \eqref{eq:adjoint2}, $v^2u=(vu)v+(vt(v))t(u)=ax=y$, and by \eqref{eq:adjoint} %and using $(\im t)^2=0$, 
we get
$\Bup(u^2,v^2)=\Bup(u,v^2u)=\Bup(u,y)=1$. Finally, $\Bup(u^2,a)=\Bup(u,au)=\Bup(u,v)=0$ 
by \eqref{eq:adjoint}, so $u^2\not\in\FF b$ 
since $\Bup(a,b)=1$. Once again, by simmetry, $v^2\not\in\FF a$.

\smallskip

\item $av^2=v^2a=0$ and $bu^2=u^2b=0$:

First, $v^2a\in\ker t\cap U_{21}$, but by \eqref{eq:adjoint2}, 
\[
v^2a=(va)v+\Bup(v,a)v+(vt(v))t(a)\in\FF v
\] 
as $va\in\FF e_2$ and $a\in\ker t$. 
So $v^2a\in\FF v\cap \ker t=0$. Once again,
by \eqref{eq:adjoint2}, $av^2=v^2a+\Bup(a,v^2)1+t(v^2)t(a)=0$, because $\Bup\vert_{U_{12}}$ is trivial.
Everything works symmetrically for $bu^2$.
\end{itemize}

\medskip

\noindent\textbf{6.}\quad Finally, we get a contradiction as follows, the subspace 
$S=\FF e_1\oplus \FF e_2\oplus \FF a\oplus \FF b\oplus \FF v^2\oplus \FF u^2$, which is contained
in $\ker t$, is a subalgebra since 
$a,v^2\in \ker t\cap U_{12}$ and $b,u^2\in \ker t\cap U_{21}$, 
$ab,au^2,v^2b,v^2u^2,ba,bv^2,u^2a,u^2v^2\in\FF e_1\oplus \FF e_2$,  $a^2,(v^2)^2,b^2,(u^2)^2\in\FF 1$
by \eqref{eq:conic}, and $av^2=v^2a=bu^2=u^2b=0$. The quadratic form $\Qup$ is multiplicative on $S$ 
(since it is so on $\ker t$) and $\Bup$ is nondegenerate on $S$. So $(S,\mu\vert_S,\Qup\vert_S,1)$ is a 
unital composition algebra in $\vect$ with $\dim S=6$, a contradiction with
Theorem \ref{th:composition_vec}.

\smallskip

This finishes the proof.

\end{proof}

\begin{remark}
Some arguments in the above proof can be derived from the fact that $\ker t$ is a conic alternative algebra and
$\ker t/\im t$ is a unital composition algebra in $\vect$, so that the results of \cite{McCrimmon} apply. In particular,
\cite[5.6 Corollary]{McCrimmon} proves $\im t=0$ if $\ker t/\im t$ is an octonion algebra,  $(\im t)^2=0$
if $\ker t/\im t$ is a quaternion algebra, and $(\im t)^3=0$ if $\ker t/\im t$ has dimension $2$.

We have preferred to give a more self-contained proof.
\end{remark}

%----------------------------------------

\end{document}